\documentclass[draft]{amsart}
\usepackage{amssymb,amsfonts,amsmath,amsthm,cite}
\usepackage{enumerate}
\usepackage{ytableau}
\ytableausetup{boxsize=2em}
\usepackage{graphicx, xcolor}
\usepackage{hyperref}
\hypersetup{colorlinks=true,linkcolor=blue,filecolor=magenta,urlcolor=cyan}
\usepackage{tikz}

\numberwithin{equation}{section}

\allowdisplaybreaks[4]

\theoremstyle{plain}
\newtheorem{theorem}{Theorem}[section]
\newtheorem{lemma}[theorem]{Lemma}
\newtheorem{corollary}[theorem]{Corollary}
\newtheorem{definition}[theorem]{Definition}

\title{A MacMahon Analysis View of 4 Diagonal DSPPs}
\author{Runqiao Li}
\address{[R.L.] University of Texas Rio Grande Valley, School of Mathematical and Statistical Sciences, Edinburg, Texas 78541, United States}
\email{runqiao.li@utrgv.edu or runqiaoli@outlook.com}

\author{Ali K. Uncu}
\address{[A.K.U.] University of Bath, Faculty of Science, Department of Computer Science, Bath, BA2 7AY, UK}
\email{aku21@bath.ac.uk}

  

\date{\today}

\begin{document}

\begin{abstract}
   We study skew double-shifted plane partitions with three-element profiles using MacMahon’s partition analysis. We present new generating function formulas for these partitions, incorporating an extra bound on the number of non-zero diagonal elements. These objects are closely related to the G\"ollnitz–Gordon and little G\"ollnitz identities. Moreover, we investigate some infinite hierarchies of $q$-series identities that stem from our formulas. We also observe certain palindromic properties within these generating functions. Finally, we examine these same objects from the perspective of linear partitions, which leads to interesting new $q$-series identities.
\end{abstract}

\maketitle

\section{Introduction}\label{sec:Intro}

An \textit{integer partition} $\lambda$ is a weakly decreasing finite list of positive integers. The elements that appear in a partition are called \textit{parts}, and their total is called the \textit{size} or the \textit{weight} of the partition, denoted by $|\lambda|$. For example, the empty list is the only partition of 0, and $(3,2,1)$ is a partition of size 6 into 3 distinct parts.

\textit{Partition identities} are statements between two different sets of partitions where the number of partitions of any size $n_1$ in one set is equal to the number of partitions of size $n_2$ in the other set. For almost all partition identities $n_1=n_2$. These identities are combinatorially interesting. These identities are highly interesting for $q$-series research too. For most partition identities, a combinatorial proof purely using bijections between sets of partitions is hard to come by. Instead, we tend to look at the generating functions for the number of partitions. These are basic special functions with a lot of added structure. One usually establishes such identities using $q$-series tools and summation formulas for basic sums. Studying partition identities and refining the generating functions with possibly more statistics/catalytic variables leads to deeper $q$-series identities and vice versa.

The G\"ollnitz--Gordon and little G\"ollnitz partition identities play a central role in the theory of partitions. They appear in many contexts and have been proven in different ways over the years (to list a few, see \cite{A, AlladiGG, AAG, BridgesUncu, Li, Berkovich_McCoy_Orrick}).

\begin{theorem}[G\"ollnitz--Gordon Identities] \label{thm_GG}
    For $i=1,\ 2$ and any non-negative integer $n$, the number of partitions of $n$ into parts $\geq 2i-1$ with gaps between parts $\geq 2$ and no consecutive even parts is equal to the number of partitions into parts congruent to $2+(-1)^i$, $4$, $6-(-1)^i$ modulo $8$.
\end{theorem}

\begin{theorem}[Little G\"ollnitz Identities]\label{thm_lG}
    For $i=1,\ 2$ and any non-negative integer $n$, the number of partitions of $n$ into parts $\geq i$ with gaps between parts $\geq 2$ and no consecutive odd parts is equal to the number of partitions into parts congruent to $i$, $4-(-1)^i$, $5+i$ modulo $8$.
\end{theorem}

Let $G_i$ be the set of partitions such that parts $\geq 2i-1$, gaps between parts $\geq 2$, and no consecutive even parts. Let $G'_i$ be the set of partitions such that parts $\geq i$, gaps between parts $\geq 2$, and no consecutive odd parts. Then, the G\"ollnitz-Gordon identities and the little G\"ollnitz identities have the following analytic forms.
\begin{align}
\sum_{\lambda\in G_1}q^{|\lambda|}=\sum_{n=0}^{\infty}\frac{q^{n^{2}}(-q;q^2)_n}{(q^2;q^2)_n}=\frac{1}{(q,q^4,q^7;q^8)_{\infty}},\\
\label{G2_prod}
\sum_{\lambda\in G_2}q^{|\lambda|}=\sum_{n=0}^{\infty}\frac{q^{n^{2}+2n}(-q;q^2)_n}{(q^2;q^2)_n}=\frac{1}{(q^3,q^4,q^5;q^8)_{\infty}},\\
\sum_{\lambda\in G'_1}q^{|\lambda|}=\sum_{n=0}^{\infty}\frac{q^{n^2+n}(-q^{-1};q^2)_{n}}{(q^2;q^2)_{n}}=\frac{1}{(q,q^5,q^6;q^8)_{\infty}},
\\\label{eq:QseriesG2}
\sum_{\lambda\in    G'_2}q^{|\lambda|}=\sum_{n=0}^{\infty}\frac{q^{n^2+n}(-q;q^2)_{n}}{(q^2;q^2)_{n}}=\frac{1}{(q^2,q^3,q^7;q^8)_{\infty}}.
\end{align}
Here and in the sequel, we adopt the standard $q$-Pochhammer symbol defined by
$$(a;q)_{0}:=1,\ (a;q)_n:=\prod_{i=0}^{n-1}(1-aq^i)\ \text{for $n\geq1$},\ \text{and}\ (a;q)_{\infty}=\lim_{n\to\infty}\prod_{i=0}^{n}(1-aq^i),$$
where $a$ and $q$ are complex variables, and the last expression converges when $|q|<1$. We shall also adopt the following abbreviations.
$$(a_1,a_2,\ldots,a_r;q)_n:=\prod_{i=1}^{r}(a_i;q)_n\ \text{and}\ (a_1,a_2,\ldots,a_r;q)_{\infty}:=\prod_{i=1}^{r}(a_i;q)_{\infty}.$$
And, the $q$-binomial coefficient is defined as
$${n\brack n}_q:=\left\{\begin{array}{cc}
   \frac{(q;q)_n}{(q;q)_m(q;q)_{n-m}}  & \text{if $n\geq m\geq0$,}  \\
    0 & \text{otherwise}. 
\end{array}\right.$$

Next, we introduce a partial order on the set of partitions. Giving two partition $\lambda=(\lambda_1,\lambda_2,\lambda_3,\ldots)$ and $\mu=(\mu_1,\mu_2,\mu_3,\ldots)$, we say $\lambda\succeq\mu$ if $\lambda_1\geq\mu_1\geq\lambda_2\geq\mu_2\geq\cdots$ and $\lambda\preceq\mu$ if $\mu_1\geq\lambda_1\geq\mu_2\geq\lambda_2\geq\cdots$.
\begin{definition}
A {\it skew double shifted plane partition} (DSPP) of {\it width} $h$ with {\it profile} $\delta=\delta_1\delta_2 \dots\delta_h=(\delta_1, \dots, \delta_h) \in \{0, 1\}^h$ is an $(h+1)$-tuple of integer partitions $(\lambda^0, \dots, \lambda^h)$ such that $\lambda^{j-1} \succeq \lambda^j$ (resp. $\lambda^{j-1} \preceq \lambda^j$) if $\delta_j =0$ (resp. if $\delta_j=1$).  We define the {\it size} as
$$
|\lambda|:=\sum_{j=0}^{h} |\lambda^j|.
$$
\end{definition}

Let $DSPP_\delta$ be the set of all DSPPs with the profile $\delta$. We define the \textit{diagonal length} of a DSPP as the maximum of the number of parts of the partitions that appear in it. 

For example, the DSPP with profile 010 and integer partitions $( (2,2), (2,1),(3,2,1),(2,2,1))$ of size 18 and diagonal length 3 can be represented as the Young tableaux:
\begin{center}
\begin{ytableau}
\none & 3 & 2\\
2 & 2 & 2 & 2\\
\none & 2 & 1 & 1 & 1
\end{ytableau},
\end{center}
where the partitions are placed on the diagonals of the DSPP. The empty DSPP is the unique DSPP with size 0 and diagonal length 0, and it belongs to all the sets $DSPP_\delta$.

Recently, Bridges and the second author \cite{BridgesUncu} showed that the product generating functions of G\"ollnitz--Gordon and little G\"ollnitz identities appear together in the same family of DSPPs that has 4 diagonals in their shifted Young tableaux. In particular, they proved that the following identities hold using Han--Xiang's product generating function formulas \cite{HX}.

Let $F_{abc}(q)$ be the generating function for the number of DSPPs with the profile $abc$, where these objects will be defined in the next section in detail; then we have the following product generating function formulas.
\begin{align}
\label{F010_product}
F_{010}(q):=&\sum_{\lambda\in\text{DSPP}_{010}}q^{|\lambda|}=\frac{1}{(q;q)_{\infty}(q,q^4,q^7;q^8)_{\infty}},\\
\label{F000_product}
F_{000}(q):=&\sum_{\lambda\in\text{DSPP}_{000}}q^{|\lambda|}=\frac{1}{(q;q)_{\infty}(q^3,q^4,q^5;q^8)_{\infty}}.
\intertext{Equations \eqref{F010_product} and \eqref{F000_product} show that the DSPP generating functions with profiles $010$ and $000$ are related to the first and second Göllnitz-Gordon identities with an extra $(q;q)^{-1}_\infty$ term, respectively.}
\label{DSPP_littleG1}
F_{001}(q):=&\sum_{\lambda\in\text{DSPP}_{001}}q^{|\lambda|}=\frac{1}{(q;q)_{\infty}(q,q^5,q^6;q^8)_{\infty}},\\
\label{DSPP_littleG2}
F_{100}(q):=&\sum_{\lambda\in\text{DSPP}_{100}}q^{|\lambda|}=\frac{1}{(q;q)_{\infty}(q^2,q^3,q^7;q^8)_{\infty}},    
\end{align}
This shows that the DSPP generating functions with profiles $001$ and $100$ are related to the first and second little Göllnitz identities with an $(q;q)^{-1}_\infty$ term, respectively. 

We study (and prove) these identities again. This time through the lens of MacMahon partition analysis, as done in \cite{LiUncu}, to discover refined generating functions for partitions that satisfy the gap conditions. Doing so also leads to the discovery of infinite hierarchies of $q$-series identities using the Bailey machinery. 

\begin{theorem}\label{main_thm} For any non-negative integer $p$,
\begin{align}
    \nonumber&\sum_{\substack{i,j,k\geq 0\\m_p\geq \dots\geq m_{1}\geq 0}} \frac{q^{2m_p^2+2m_{p-1}^2+\dots +2m_1^2+i^2+j^2 + k^2} }{(q^2;q^2)_{m_p - m_{p-1}}\dots (q^2;q^2)_{m_2-m_1}(q^2;q^2)_{2m_1}}  {m_1 \brack i}_{q^2} { i - k\brack j}_{q^2} {j\brack k}_{q^2} \\ \label{eq_double_baileyd_P010lim}&\hspace{1cm}= \frac{(q^{40+64p};q^{40+64p})_\infty}{(q^2;q^2)_\infty}\left( (q^{19+32p}, q^{21+32p};q^{40+64p})_\infty + q^{1+2p}(q^{11+16p}, q^{29+48p};q^{40+64p})_\infty
    \right),\\\nonumber\\
    \nonumber&\sum_{\substack{i,j,k\geq 0\\m_p\geq \dots\geq m_{1}\geq 0}} \frac{q^{2m_p^2+2m_{p-1}^2+\dots +2m_1^2+i^2+j^2 + k^2+2j} }{(q^2;q^2)_{m_p - m_{p-1}}\dots (q^2;q^2)_{m_2-m_1}(q^2;q^2)_{2m_1}}  {m_1 \brack i}_{q^2} { i - k-1\brack j}_{q^2} {j\brack k}_{q^2} \\ \label{eq_double_baileyd_P000lim}& \hspace{1cm}=q^{1+2p}\frac{(q^{40+64p};q^{40+64p})_\infty}{(q^2;q^2)_\infty}\left( (q^{9+16p}, q^{31+48p};q^{40+64p})_\infty + q^{3(1+2p)}(q, q^{39+64p};q^{40+64p})_\infty
    \right),\\
    \nonumber\\
    \nonumber&\sum_{\substack{i,j,k\geq 0\\m_p\geq \dots\geq m_{1}\geq 0}} \frac{q^{2m_p^2+2m_{p-1}^2+\dots 2m_2^2+4m_1^2-2im_1+i^2-2ij+j^2 + k^2} }{(q^2;q^2)_{m_p - m_{p-1}}\dots (q^2;q^2)_{m_2-m_1}(q^2;q^2)_{2m_1}} {m_1 \brack i}_{q^2} { i - k\brack j}_{q^2} {j\brack k}_{q^2} \\ \label{eq_double_baileyd_P010_dual_lim}&\hspace{1cm}=\frac{(q^{24+64p};q^{24+64p})_\infty}{(q^2;q^2)_\infty}\left( (q^{11+32p}, q^{13+32p};q^{24+64p})_\infty + q^{1+2p}(q^{5+16p}, q^{19+48p};q^{24+64p})_\infty
    \right),\\
    \nonumber\\
    \nonumber&\sum_{\substack{i,j,k\geq 0\\m_p\geq \dots\geq m_{1}\geq 0}} \frac{q^{2m_p^2+2m_{p-1}^2+\dots 2m_2^2+4m_1^2-2im_1+i^2-2ij+j^2 + k^2} }{(q^2;q^2)_{m_p - m_{p-1}}\dots (q^2;q^2)_{m_2-m_1}(q^2;q^2)_{2m_1}} {m_1 \brack i}_{q^2} { i - k-1\brack j}_{q^2} {j\brack k}_{q^2} \\ \label{eq_double_baileyd_P000_dual_lim}&\hspace{1cm}=q^{1+2p}\frac{(q^{24+64p};q^{24+64p})_\infty}{(q^2;q^2)_\infty}\left( (q^{7+16p}, q^{17+48p};q^{24+64p})_\infty - q^{3(1+2p)-1}(q, q^{23+64p};q^{24+64p})_\infty
    \right).
\end{align}
\end{theorem}

The rest of this paper is organized as follows. In Section~\ref{sec:PA} we give a brief introduction of MacMahon's Partition Analysis, and we deduce the crude form of $F_{010}(n)$ as an example. Section~\ref{sec:010},\,\ref{sec:000},\ref{sec:001}, and \ref{sec:100} are devoted to DSPP's with profile $010$, $000$, $001$, and $100$, respectively. In Section~\ref{sec:Palindromity}, we discuss the palindromic property of those finite generating functions and introduce a more natural way to cut the DSPP's with profile $100$. In Section~\ref{sec:LinearPartition}, we discuss the connection between DSPP's with bounded diagonal length and linear partitions related to G\"ollnitz--Gordon and little G\"ollnitz identities. Finally, we give some conclusive remarks in Section~\ref{sec:Conclusions}.

\section{Partition Analysis}\label{sec:PA}

Partition analysis was first introduced by MacMahon in his Combinatorial Analysis \cite{MacMahon}. It relies on the Omega operator, which was defined as follows.
\begin{definition}
The Omega operator $\Omega_{\geq}$ is given by
$$\underset{\geq}{\Omega}\sum_{s_1=-\infty}^{\infty}\cdots\sum_{s_r=-\infty}^{\infty}A_{s_1,\ldots,s_r}\lambda_1^{s_1}\cdots\lambda_r^{s_r}:=\sum_{s_1=0}^{\infty}\cdots\sum_{s_r=0}^{\infty}A_{s_1,\ldots,s_r},$$
where the domain of the $A_{s_1\ldots,s_r}$ is the field of rational functions over $\mathbb{C}$
in several complex variables and the $\lambda_i$ are restricted to a neighborhood
of the circle $|\lambda_i|=1$.
\end{definition}

This operator gives an efficient way to compute generating functions for various partitions, especially with gap conditions. Andrews and Paule have written a series of papers on Partition Analysis. There are also recent works related to it. See \cite{AndrewsI,AndrewsPauleVIII,AndrewsPauleXIII,Li,LiUncu} for example.

We will mainly need this operator to produce initial values for the qFunctions package in order to guess a recurrence before proofs. Here we demonstrate this process with DSPPs of profile $010$. The following diagram gives a partition in $\text{DSPP}_{010}$ with diagonal length bounded by $n$.
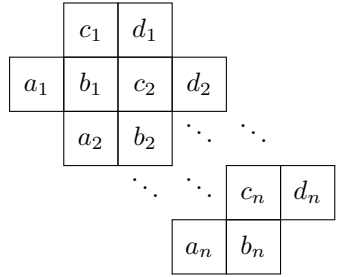
\begin{center}
\begin{ytableau}
\none & c_1 & d_1\\
a_1 & b_1 & c_2 & d_2\\
\none & a_2 & b_2 & \none[\ddots] & \none[\ddots]\\
\none & \none & \none[\ddots] & \none[\ddots] & c_{n} & d_{n}\\
\none & \none & \none & a_{n} & b_{n}
\end{ytableau}
\end{center}
To apply the Omega operator, we need to convert the weakly decreasing condition on each row and column into inequalities and encode them as the exponents of some extra variables. For this diagram, we need the following inequalities for $n\geq i\geq1$.
$$a_i-b_i\geq0,\ b_{i}-c_{i+1}\geq0,\ c_{i}-d_{i}\geq0,\ d_i\geq0,$$
and
$$a_i\geq0,\ b_{i}-a_{i+1}\geq0,\ c_{i}-b_{i}\geq0,\ d_i-c_{i+1}\geq0.$$
Here we take $a_i=b_i=c_i=d_i=0$ for $i>n$. The inequalities in the first row indicate that the rows are weakly decreasing, while those in the second row indicate the columns are weakly decreasing. We shall encode them as exponents of $\lambda$'s and $\mu$'s, respectively.

Thus, define
$$F_{c}(n):=\sum_{\pi\in\text{DSPP}_{c}(n)}q^{|\pi|}$$
as the generating function for DSPPs with profile $c$ and diagonal length bounded by $n$. The crude form of $F_{010}(n)$ would be
\begin{align*}
F_{010}&(n)=\underset{\geq}{\Omega}\sum_{\substack{a_i,b_i,c_i,d_i\geq0\\\text{for}\ n\geq i\geq1}}q^{\sum_{i=1}^{n}a_i+b_i+c_i+d_i}\prod_{i=1}^{n}\lambda_{i,i+1}^{c_i-d_i}\lambda_{i,i+2}^{d_i}\lambda_{i+1,i}^{a_i-b_i}\lambda_{i+1,i+1}^{b_i-c_{i+1}}\mu_{i,i+1}^{c_i-b_i}\mu_{i,i+2}^{d_i-c_{i+2}}\mu_{i+1,i}^{a_i}\mu_{i+1,i+1}^{b_i-a_{i+1}}\\
=&\underset{\geq}{\Omega}\sum_{a_1\geq0}q^{a_1}\lambda_{2,1}^{a_1}\mu_{2,1}^{a_1}\sum_{c_1\geq0}q^{c_1}\lambda_{1,2}^{c_1}\mu_{1,2}^{c_1}\prod_{i=1}^{n}\left(\sum_{a_i\geq0}q^{a_i}\lambda_{i+1,i}^{a_i}\mu_{i+1,i}^{a_i}\mu_{i,i}^{-a_i}\sum_{c_i\geq0}q^{c_i}\lambda_{i,i+1}^{c_i}\mu_{i,i+1}^{c_i}\lambda_{i,i}^{-c_i}\mu_{i-1,i+1}^{-c_i}\right)\\
&\times\prod_{i=1}^{n}\left(\sum_{b_i\geq0}q^{b_i}\lambda_{i+1,i+1}^{b_i}\mu_{i+1,i+1}^{b_i}\lambda_{i+1,i}^{-b_i}\mu_{i,i+1}^{-b_i}\sum_{d_i\geq0}q^{d_i}\lambda_{i,i+2}^{d_i}\mu_{i,i+2}^{d_i}\lambda_{i,i+1}^{-d_i}\right)\\
=&\underset{\geq}{\Omega}\frac{1}{(1-q\lambda_{2,1}\mu_{2,1})(1-q\lambda_{1,2}\mu_{1,2})}\prod_{i=2}^{n}\left(1-\frac{q\lambda_{i+1,i}\mu_{i+1,i}}{\mu_{i,i}}\right)^{-1}\left(1-\frac{q\lambda_{i,i+1}\mu_{i,i+1}}{\lambda_{i,i}\mu_{i-1,i+1}}\right)^{-1}\\
&\times\prod_{i=1}^{n}\left(1-\frac{q\lambda_{i+1,i+1}\mu_{i+1,i+1}}{\lambda_{i+1,i}\mu_{i,i+1}}\right)^{-1}\left(1-\frac{q\lambda_{i,i+2}\mu_{i,i+2}}{\lambda_{i,i+1}}\right)^{-1},
\end{align*}
where we take $a_i=b_i=c_i=d_i=0$ for $i>n$ in the summation in first row.

This generating function is called the crude form because of those extra variables $\lambda$ and $\mu$. The next step would be applying certain elimination rules to cancel them. In \cite{AndrewsI}, Andrews listed several frequently used rules for the Omega operator. One can also derive rules by the patterns in the crude form, like the author did in \cite{Li}. For our purpose, we shall need the following.
\begin{lemma}The following elimination rules hold for the Omega operator.
\begin{equation}\label{eq:Elimination1}
\underset{\geq}{\Omega}\frac{1}{(1-A\lambda)}=\frac{1}{(1-A)},
\end{equation}
\begin{equation}\label{eq:Elimination2}
\underset{\geq}{\Omega}\frac{1}{(1-A\lambda)(1-\frac{B}{\lambda})}=\frac{1}{(1-A)(1-AB)},   
\end{equation}
\begin{equation}\label{eq:Elimination3}
\underset{\geq}{\Omega}\frac{1}{(1-A\lambda)(1-B\lambda)(1-\frac{C}{\lambda})}=\frac{1-ABC}{(1-A)(1-B)(1-AC)(1-BC)}.
\end{equation}

\begin{proof}
Here we shall prove the last one, for the first two are corollaries of it.
\begin{align*}
\underset{\geq}{\Omega}\frac{1}{(1-A\lambda)(1-B\lambda)(1-\frac{C}{\lambda})}=&\underset{\geq}{\Omega}\sum_{i\geq0}(A\lambda)^{i}\sum_{j\geq0}(B\lambda)^j\sum_{k\geq0}(C/\lambda)^k\\
=&\underset{\geq}{\Omega}\sum_{i,j,k\geq0}A^{i}B^j C^{k}\lambda^{i+j-k}\\
=&\sum_{i,j\geq0}A^{i}B^{j}\sum_{k=0}^{i+j}C^k\\
=&\frac{1}{1-C}\left(\sum_{i,j\geq0}A^iB^j-C\sum_{i\geq0}A^iC^{i}\sum_{j\geq0}B^jC^j\right)\\
=&\frac{(1-AC)(1-BC)-C(1-A)(1-B)}{(1-C)(1-A)(1-B)(1-AC)(1-BC)}\\
=&\frac{1-ABC}{(1-A)(1-B)(1-AC)(1-BC)}.
\end{align*}
This finishes the proof.
\end{proof}
\end{lemma}
Next, to demonstrate the elimination process, we present the computation for $F_{010}(1)$.
\begin{align*}
F_{010}(1)=&\underset{\geq}{\Omega}\frac{1}{(1-q\lambda_{2,1}\mu_{2,1})(1-q\lambda_{1,2}\mu_{1,2})}\left(1-\frac{q\lambda_{2,2}\mu_{2,2}}{\lambda_{2,1}\mu_{1,2}}\right)^{-1}\left(1-\frac{q\lambda_{1,3}\mu_{1,3}}{\lambda_{1,2}}\right)^{-1}\\
=&\underset{\geq}{\Omega}\frac{1}{(1-q\lambda_{2,1})(1-q\lambda_{1,2}\mu_{1,2})}\left(1-\frac{q}{\lambda_{2,1}\mu_{1,2}}\right)^{-1}\left(1-\frac{q}{\lambda_{1,2}}\right)^{-1}&\text{(by \eqref{eq:Elimination1})}\\
=&\underset{\geq}{\Omega}\frac{1}{(1-q)(1-q^2/\mu_{1,2})(1-q\mu_{1,2})(1-q^2\mu_{1,2})}&\text{(by \eqref{eq:Elimination2})}\\
=&\frac{1-q^5}{(1-q)^2(1-q^2)(1-q^3)(1-q^4)}&\text{(by \eqref{eq:Elimination3})}\\
=&\frac{1}{(q;q)_4}(1+q+q^2+q^3+q^4).
\end{align*}
This computation can also be done by the Mathematica package \cite{OmegaPackage}, and in practice, we usually need to compute a small number of these initial cases. For the other three profiles, we shall attain the initial values in the same way.

\section{Profile $010$}\label{sec:010}

The diagram of a partition in $\text{DSPP}_{010}$ with at most $4n$ parts can be of the following shapes.
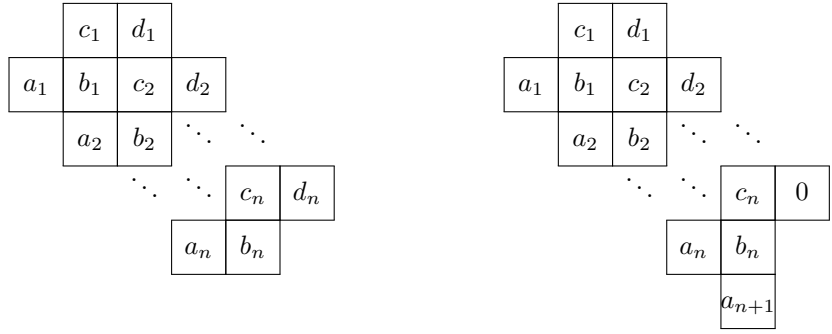
\begin{center}
\begin{ytableau}
\none & c_1 & d_1\\
a_1 & b_1 & c_2 & d_2\\
\none & a_2 & b_2 & \none[\ddots] & \none[\ddots]\\
\none & \none & \none[\ddots] & \none[\ddots] & c_{n} & d_{n}\\
\none & \none & \none & a_{n} & b_{n}
\end{ytableau}
\hspace{2cm}
\begin{ytableau}
\none & c_1 & d_1\\
a_1 & b_1 & c_2 & d_2\\
\none & a_2 & b_2 & \none[\ddots] & \none[\ddots]\\
\none & \none & \none[\ddots] & \none[\ddots] & c_{n} & 0\\
\none & \none & \none & a_{n} & b_{n}\\
\none & \none & \none & \none & a_{n+1}
\end{ytableau}
\end{center}

Let $\text{DSPP}_{010}(n)$ be the set of double shifted plane partitions with diagonal length bounded by $n$, that is, $a_n+1=b_{n+1}=c_{n+1}=d_{n+1}=0$. And, let $\text{DSPP}'_{010}(n)$ be the set of double shifted plane partitions such that $a_{n+1}>0$ and $b_{n+1}=c_{n+1}=d_{n}=0$. We are mainly interested in $\text{DSPP}_{010}(n)$, but we need both of them to complete the recurrence argument. Let
$$F_{010}(n):=\sum_{\Lambda\in\text{DSPP}_{010}(n)}q^{|\Lambda|},\quad F'_{010}(n):=\sum_{\Lambda\in\text{DSPP}'_{010}(n)}q^{|\Lambda|}\quad\text{and}\quad F^{T}_{010}(n):=\sum_{\Lambda\in\text{DSPP}^{T}_{010}(n)}q^{|\Lambda|},$$
where $\text{DSPP}^{T}_{010}(n)=\text{DSPP}_{010}(n)\cup\text{DSPP}'_{010}(n)$ is the set all double shifted partitions with profile $010$ and have at most $4n$ parts. Note that $\text{DSPP}_{010}(n)\cap\text{DSPP}'_{010}(n)=\emptyset$, so $F^{T}_{010}(n)=F_{010}(n)+F'_{010}(n)$.

Recall that in Section \ref{sec:PA} we have seen the following crude generating function.
\begin{theorem}
\begin{equation}
\begin{split}
F_{010}(n)=&\underset{\geq}{\Omega}\frac{1}{(1-q\lambda_{2,1}\mu_{2,1})(1-q\lambda_{1,2}\mu_{1,2})}\prod_{i=2}^{n}\left(1-\frac{q\lambda_{i+1,i}\mu_{i+1,i}}{\mu_{i,i}}\right)^{-1}\left(1-\frac{q\lambda_{i,i+1}\mu_{i,i+1}}{\lambda_{i,i}\mu_{i-1,i+1}}\right)^{-1}\\
&\times\prod_{i=1}^{n}\left(1-\frac{q\lambda_{i+1,i+1}\mu_{i+1,i+1}}{\lambda_{i+1,i}\mu_{i,i+1}}\right)^{-1}\left(1-\frac{q\lambda_{i,i+2}\mu_{i,i+2}}{\lambda_{i,i+1}}\right)^{-1}.
\end{split}
\end{equation}    
\end{theorem}

By evaluating this for small $n$, we have the following initial values.
\begin{theorem}\label{thm:RecF010}
\begin{equation*}
F_{010}(1)=\frac{1}{(q;q)_4}(1+q+q^2+q^3+q^4),
\end{equation*}
 \begin{align*}
    F_{010}(2)=&\frac{1}{(q;q)_{8}}(1+q+q^2+q^3+2q^4+2q^5+2q^6+3q^7+3q^8+3q^9+2q^{10}\\   &+2q^{11}+2q^{12}+q^{13}+q^{14}+q^{15}+q^{16}).   
    \end{align*}
\end{theorem}

\begin{theorem}\label{thm:F010Rec}
For any $n\geq1$,
\begin{equation}\label{eq:F010Rec}
F_{010}(n)=\frac{1+q^{4n-1}+q^{4n-2}+q^{4n-3}+q^{8n-4}}{(q^{4n-3};q)_{4}}F_{010}(n-1)+\frac{1+q^{4n-1}}{(q^{4n-3};q)_{4}}F'_{010}(n-1),
\end{equation}
\begin{equation}\label{eq:F'010Rec}
F'_{010}(n)=\frac{q^{4n}(1+q^{4n-3})}{(q^{4n-3};q)_{4}}F_{010}(n-1)+\frac{q^{4n}}{(q^{4n-3};q)_{4}}F'_{010}(n-1).
\end{equation}  
\end{theorem}

\begin{proof}
We start with~\eqref{eq:F010Rec}. Given a double shifted plane partition $\pi\in\text{DSPP}_{010}(n)$, we have the following two cases.

\textbf{Case 1.} $a_n\leq d_{n-1}$. In this case, we need to get rid of $a_n,b_n,c_n$ and $d_n$ from the diagram, and we have the following subcases.
\begin{itemize}
    \item[(i)] $c_n\geq d_n\geq a_n\geq b_n$.
In this case, we do the following operations.
\begin{align*}
&\begin{pmatrix}
 & c_1 & d_1 & \\
a_1 & b_1 & \ddots & \ddots\\
 & \ddots & \ddots &c_{n-1} & d_{n-1}\\
& & a_{n-1} & b_{n-1} & c_n & d_n\\
& & & a_n & b_n 
\end{pmatrix}\\
\longrightarrow
&\begin{pmatrix}
 & c_1-b_n & d_1-b_n & \\
a_1-b_n & b_1-b_n & \ddots & \ddots\\
 & \ddots & \ddots &c_{n-1}-b_n & d_{n-1}-b_n\\
& & a_{n-1}-b_n & b_{n-1}-b_n & c_n-b_n & d_n-b_n\\
& & & a_n-b_n & 0 
\end{pmatrix}\\
\longrightarrow
&\begin{pmatrix}
 & c_1-a_n & d_1-a_n & \\
a_1-a_n & b_1-a_n & \ddots & \ddots\\
 & \ddots & \ddots &c_{n-1}-a_n & d_{n-1}-a_n\\
& & a_{n-1}-a_n & b_{n-1}-a_n & c_n-a_n & d_n-a_n\\
& & & 0 & 0
\end{pmatrix}\\
\longrightarrow
&\begin{pmatrix}
 & c_1-d_n & d_1-d_n & \\
a_1-d_n & b_1-d_n & \ddots & \ddots\\
 & \ddots & \ddots &c_{n-1}-d_n & d_{n-1}-d_n\\
& & a_{n-1}-d_n & b_{n-1}-d_n & c_n-d_n & 0\\
& & & 0 & 0
\end{pmatrix}\\
\longrightarrow
&
\begin{pmatrix}
 & c_1-c_n & d_1-c_n & \\
a_1-c_n & b_1-c_n & \ddots & \ddots\\
 & \ddots & \ddots &c_{n-1}-c_n & d_{n-1}-c_n\\
& & a_{n-1}-c_n & b_{n-1}-c_n & 0 & 0\\
& & & 0 & 0
\end{pmatrix}
\end{align*}
We have subtracted $b_n$, $a_n-b_n$, $d_n-a_n$, and $c_n-d_n$ from the remaining parts in each step, and we end up with a partition in $\text{DSPP}_{010}(n-1)$. The weight being subtracted are generated by $\frac{1}{(q^{4n-3};q)_{4}}$.
    \item[(ii)] $c_n\geq a_n>d_n\geq b_n$. In this case, we first subtract $1$ from every part except $b_n$ and $d_n$ and then swap $a_n-1$ and $d_n$.
\begin{align*}
\begin{pmatrix}
 & c_1 & d_1 & \\
a_1 & b_1 & \ddots & \ddots\\
 & \ddots & \ddots &c_{n-1} & d_{n-1}\\
& & a_{n-1} & b_{n-1} & c_n & d_n\\
& & & a_n & b_n 
\end{pmatrix}
\longrightarrow
\begin{pmatrix}
 & c_1-1& d_1-1& \\
a_1-1& b_1-1& \ddots & \ddots\\
 & \ddots & \ddots &c_{n-1}-1& d_{n-1}-1\\
& & a_{n-1}-1& b_{n-1}-1& c_n-1 & a_n-1\\
& & & d_n & b_n 
\end{pmatrix}
\end{align*}
By this, we have subtracted the total weight of $q^{4n-2}$, and we end up with the order $c_n-1\geq a_{n}-1\geq d_n\geq b_n$, which is the same as Case (i). So, by applying the same operations as in (i), we get a partition in $\text{DSPP}_{010}(n-1)$, and the weight we subtracted is generated by $\frac{q^{4n-2}}{(q^{4n-3};q)_4}$.

    \item[(iii)] $c_n\geq a_n\geq b_n>d_n$. In this case, we first subtract $1$ from each part except $d_{n}$, then rearrange $a_n-1$, $b_{n}-1$ and $d_n$ as follows.
\begin{align*}
\begin{pmatrix}
 & c_1 & d_1 & \\
a_1 & b_1 & \ddots & \ddots\\
 & \ddots & \ddots &c_{n-1} & d_{n-1}\\
& & a_{n-1} & b_{n-1} & c_n & d_n\\
& & & a_n & b_n 
\end{pmatrix}
\longrightarrow
\begin{pmatrix}
 & c_1-1& d_1-1& \\
a_1-1& b_1-1& \ddots & \ddots\\
 & \ddots & \ddots &c_{n-1}-1& d_{n-1}-1\\
& & a_{n-1}-1& b_{n-1}-1& c_n-1 & a_n-1\\
& & & b_n-1 & d_n 
\end{pmatrix}
\end{align*}
Note that we have subtracted the total weight of $q^{4n-1}$ from $\pi$, and in the new diagram we have $c_{n}-1\geq a_n-1\geq b_{n}-1\geq d_n$. So, the rest is the same as in Case (i). And the weight we subtracted is generated by $\frac{q^{4n-1}}{(q^{4n-3};q)_4}$.
    \item[(iv)] $a_n> c_n\geq d_n\geq b_n$. In this case, we first subtract $1$ from each part except $b_n$, $c_n$ and $d_{n}$, then rearrange $a_n-1$ $c_n$ and $d_n$ as follows.
\begin{align*}
\begin{pmatrix}
 & c_1 & d_1 & \\
a_1 & b_1 & \ddots & \ddots\\
 & \ddots & \ddots &c_{n-1} & d_{n-1}\\
& & a_{n-1} & b_{n-1} & c_n & d_n\\
& & & a_n & b_n 
\end{pmatrix}
\longrightarrow
\begin{pmatrix}
 & c_1-1& d_1-1& \\
a_1-1& b_1-1& \ddots & \ddots\\
 & \ddots & \ddots &c_{n-1}-1& d_{n-1}-1\\
& & a_{n-1}-1& b_{n-1}-1& a_n-1 & c_n\\
& & & d_n & b_n 
\end{pmatrix}
\end{align*}
Note that we have subtracted the total weight of $q^{4n-3}$, and in the new diagram we have $a_{n}-1\geq c_n\geq d_{n}\geq b_n$. So, the rest is the same as in Case (i). And the weight we subtracted is generated by $\frac{q^{4n-3}}{(q^{4n-3};q)_4}$.

    \item[(v)] $a_n> c_n\geq b_n>d_n$.
In this case, we first subtract $1$ from each part except $d_n$ and then swap $b_n-1$ and $d_n$.
\begin{align*}
\begin{pmatrix}
 & c_1 & d_1 & \\
a_1 & b_1 & \ddots & \ddots\\
 & \ddots & \ddots &c_{n-1} & d_{n-1}\\
& & a_{n-1} & b_{n-1} & c_n & d_n\\
& & & a_n & b_n 
\end{pmatrix}
\longrightarrow
\begin{pmatrix}
 & c_1-1& d_1-1& \\
a_1-1& b_1-1& \ddots & \ddots\\
 & \ddots & \ddots &c_{n-1}-1& d_{n-1}-1\\
& & a_{n-1}-1& b_{n-1}-1& c_n-1 & b_n-1\\
& & & a_n-1 & d_n 
\end{pmatrix}
\end{align*}
We have subtracted the total weight of $q^{4n-1}$, and in the resulting diagram, we have $a_n-1\geq c_n-1\geq b_{n}-1\geq d_n$, which is the same as in Case (iv), so we only need to apply the same operations. So, the weight we subtracted in this case is generated by $q^{4n-1}\times\frac{ q^{4n-3}}{(q^{4n-3};q)_4}=\frac{q^{8n-4}}{(q^{4n-3};q)_4}$.
\end{itemize}

Summing up all the subcases in \textbf{Case $1.$}, we get
$$\frac{1+q^{4n-1}+q^{4n-2}+q^{4n-3}+q^{8n-4}}{(q^{4n-3};q)_4}F_{010}(n-1),$$
which explains the first term on the right-hand side of~\eqref{eq:F010Rec}.

\textbf{Case 2.} $a_n>d_{n-1}$. In this case, we need to get rid of $d_{n-1}$, $b_n$ $c_n$ and $d_n$ from $\pi$ to get a partition in $\text{DSPP}'_{010}(n-1)$. Again, we consider the following subcases.
\begin{itemize}
    \item[(i)] $d_{n-1}\geq c_{n}\geq d_{n}\geq b_n$. In this case, we shall, step by step, subtract $b_n$, $d_n-b_n$, $c_n-d_n$ and $d_{n-1}-c_n$ from the remind parts of $\pi$.
\begin{align*}
&\begin{pmatrix}
 & c_1 & d_1 & \\
a_1 & b_1 & \ddots & \ddots\\
 & \ddots & \ddots &c_{n-1} & d_{n-1}\\
& & a_{n-1} & b_{n-1} & c_n & d_n\\
& & & a_n & b_n 
\end{pmatrix}\\
\longrightarrow
&\begin{pmatrix}
 & c_1-b_n & d_1-b_n & \\
a_1-b_n & b_1-b_n & \ddots & \ddots\\
 & \ddots & \ddots &c_{n-1}-b_n & d_{n-1}-b_n\\
& & a_{n-1}-b_n & b_{n-1}-b_n & c_n-b_n & d_n-b_n\\
& & & a_n-b_n & 0 
\end{pmatrix}\\
\longrightarrow
&\begin{pmatrix}
 & c_1-a_n & d_1-a_n & \\
a_1-a_n & b_1-a_n & \ddots & \ddots\\
 & \ddots & \ddots &c_{n-1}-a_n & d_{n-1}-a_n\\
& & a_{n-1}-a_n & b_{n-1}-a_n & c_n-a_n & 0\\
& & & a_n-d_n & 0
\end{pmatrix}\\
\longrightarrow
&\begin{pmatrix}
 & c_1-c_n & d_1-c_n & \\
a_1-c_n & b_1-c_n & \ddots & \ddots\\
 & \ddots & \ddots &c_{n-1}-c_n & d_{n-1}-c_n\\
& & a_{n-1}-c_n & b_{n-1}-c_n & 0 & 0\\
& & & a_n-c_n & 0
\end{pmatrix}\\
\longrightarrow
&\begin{pmatrix}
 & c_1-d_{n-1} & d_1-d_{n-1} & \\
a_1-d_{n-1} & b_1-d_{n-1} & \ddots & \ddots\\
 & \ddots & \ddots &c_{n-1}-d_{n-1} & 0 \\
& & a_{n-1}-d_{n-1} & b_{n-1}-d_{n-1} & 0 & 0\\
& & & a_n-d_{n-1} & 0
\end{pmatrix}
\end{align*}
Note that the weight been subtracted is generated by $\frac{1}{(q^{4n-3};q)_{4}}$, and the resulted partition is in $\text{DSPP}_{010}'(n-1)$ as desired.
    \item[(ii)] $d_{n-1}\geq c_{n}\geq b_{n}> d_n$. In this case, we first subtract $1$ from each part of $\pi$ except $d_n$, and then swap $b_{n-1}$ and $d_n$.
\begin{align*}
\begin{pmatrix}
 & c_1 & d_1 & \\
a_1 & b_1 & \ddots & \ddots\\
 & \ddots & \ddots &c_{n-1} & d_{n-1}\\
& & a_{n-1} & b_{n-1} & c_n & d_n\\
& & & a_n & b_n 
\end{pmatrix}
\longrightarrow
\begin{pmatrix}
 & c_1-1& d_1-1& \\
a_1-1& b_1-1& \ddots & \ddots\\
 & \ddots & \ddots &c_{n-1}-1& d_{n-1}-1\\
& & a_{n-1}-1& b_{n-1}-1& c_n-1 & b_n-1\\
& & & a_n-1 & d_n 
\end{pmatrix}
\end{align*}
The weight we subtracted is $q^{4n-3}$, and the resulting diagram is the same as in Case (i). So, in this case, the weight been subtracted is generated by $\frac{q^{4n-1}}{(q^{4n-3};q)_4}$.   
\end{itemize}

Summing up the two subcases of \textbf{Case $2.$}, we have
$$\frac{1+q^{4n-1}}{(q^{4n-3};q)_4}F'_{010}(n-1),$$
which is the second term on the right-hand side of~\eqref{eq:F010Rec}. Combining \textbf{Case $1.$} and \textbf{Case $2.$}, we finished the proof for~\eqref{eq:F010Rec}. By a similar argument, \eqref{eq:F'010Rec} can also be established, so we skip it.
\end{proof}

Let $P_{010}(n):=(q;q)_{4n}F_{010}(n)$ and $P'_{010}(n):=(q;q)_{4n}F'_{010}(n)$, then Theorem~\ref{thm:RecF010} implies the following.
\begin{theorem}\label{thm:P010Rec}
For any $n\geq1$,
\begin{equation} \label{P010_rec_coupled}
P_{010}(n)=(1+q^{4n-1}+q^{4n-2}+q^{4n-3}+q^{8n-4})P_{010}(n-1)+(1+q^{4n-1})P'_{010}(n-1),
\end{equation}
\begin{equation}\label{P010p_rec_coupled}
P'_{010}(n)=q^{4n}(1+q^{4n-3})P_{010}(n-1)+q^{4n}P'_{010}(n-1).
\end{equation}      
\end{theorem}

The coupled recurrence system in Theorem~\ref{thm:P010Rec} can be uncoupled by substitution and some algebra. We do it openly for good measure. Solving \eqref{P010_rec_coupled} for $P'_{010}(n-1)$ and plugging it in \eqref{P010p_rec_coupled} yields
\begin{equation}
    \label{P010p_rec_intermediate}
    P'_{010}(n) = q^{4n}(1+q^{4n-3})P_{010}(n-1) + \frac{q^{4n}}{1-q^{4n-1}}(P_{010}(n)-(1+q^{4n-1}+q^{4n-2}+q^{4n-3}+q^{8n-4})P_{010}(n-1))
\end{equation}
Substituting \eqref{P010p_rec_intermediate} with $n\mapsto n-1$ in \eqref{P010_rec_coupled} yields an equivalent recurrence to the following:\begin{align}
     \nonumber q^{6 + 8 n} (1 + q^{7 + 4 n}) P_{010}(
   n) - (1 + q^{5 + 4 n}) (1 + q^{3 + 4 n} + q^{4 + 4 n} + q^{
    6 + 4 n} &+ q^{7 + 4 n} + q^{10 + 8 n}) P_{010}(n+1)\\ \label{P010_rec_uncoupled}&+ (1 + q^{3 + 4 n}) P_{010}( n+2) = 0
\end{align}

Similarly, the uncoupled recurrence for $P'_{010}(n)$ is 
\begin{align}
     \nonumber q^{10 + 8 n} (1 + q^{5 + 4 n}) P'_{010}(
   n) - q^{4n}(1 + q^{3 + 4 n}) (1 + q^{1 + 4 n} + q^{2 + 4 n} + q^{
    4 + 4 n} &+ q^{5 + 4 n} + q^{6 + 8 n}) P'_{010}(n+1)\\ \label{P'010_rec_uncoupled}&+ (1 + q^{1 + 4 n}) P'_{010}( n+2) = 0
\end{align}

For more complicated coupled systems, this uncoupling can be done through other means, such as Gr\"obner bases. For other applications of such uncoupling, we invite readers to see \cite{HolonomicFunctions, CDU}.

Looking at the initial terms, we can also guess a closed-form solution for this sequence using \cite{qFuncs}. We represent this as the next theorem.

\begin{theorem}\label{thm:P010_fermionic_rep}
    For $n\geq 0$, \begin{equation}
        \label{P010_fermionic_rep} P_{010}(n) = \sum_{j,k\geq 0} q^{j^2 + k^2} {2 n - k\brack j}_{q^2} {j\brack k}_{q^2}.
    \end{equation}
\end{theorem}

We prove this easily by \cite{qMultiSum}, by showing that the right-hand side sum satisfies \eqref{P010_rec_uncoupled} and has the same initial conditions as $P_{010}(n)$.

Berkovich--McCoy--Orrick \cite{Berkovich_McCoy_Orrick} showed that the right-hand side double sum is the generating function for the number of partitions into parts $\leq 4n-1$ that satisfy the gap conditions of the G\"ollnitz--Gordon partition theorem. An alternative double sum representation with $q$-binomial bases $q$ and $q^4$ rather than uniformly $q^2$ as in \eqref{P010_rec_uncoupled} is given in \cite{doublesums}. 

Now that a clear connection of $P_{010}(n)$ with the G\"ollnitz-Gordon identities is established, we can directly prove \eqref{F010_product} by first observing \[F_{010}(n) = \frac{1}{(q;q)_{4n}} \sum_{j,k\geq 0} q^{j^2 + k^2} {2 n - k\brack j}_{q^2} {j\brack k}_{q^2}\] and then replacing the sum generating function of the G\"ollnitz--Gordon partitions with the relevant product generating function in the limit $n\rightarrow\infty$. We note that another proof of \eqref{F010_product} was given in \cite[Section 4.3]{BridgesUncu} through the functional equations and not through first establishing a bounded generating function. 

Similarly, we can also prove the following theorem about $P'_{010}(n)$.
\begin{theorem}\label{thm:P010p_fermionic_rep}
    For $n\geq 0$, \begin{align}
        \label{P010p_fermionic_rep} P'_{010}(n) &= q^{4n} \sum_{j,k\geq 0} q^{j^2 + k^2} {2 n - k-1\brack j}_{q^2} {j\brack k}_{q^2},
        \intertext{and} \label{F010p_fermionic_rep}F'_{010}(n) &= P'_{010}(n)/ (q;q)_{4n}.
    \end{align}
\end{theorem}

In \cite{Berkovich_McCoy_Orrick}, Berkovich--McCoy--Orrick also found a polynomial relation that yields the first G\"ollnitz--Gordon identity directly:

\begin{theorem}[Berkovich--McCoy--Orrick] For every non-negative integer $L$,
\begin{equation}
    \label{BMO_GG1_bosonic}
    \sum_{j,k\geq 0} q^{j^2 + k^2} {L - k\brack j}_{q^2} {j\brack k}_{q^2} = \sum_{j=-\infty}^\infty (-1)^j q^{4j^2+j} \left(T_0(L,4j,q^2) +  T_0(L,4j+1,q^2) \right),
\end{equation}
where the trinomial coefficient is defined as \[T_0(n,j,q) := \sum_{r=0}^n (-1)^r {n \brack r}_{q^2}{2n-2r\brack n-j-r}_{q}.\]
\end{theorem}

The identity \eqref{BMO_GG1_bosonic} lies on an infinite hierarchy also noted in \cite{Berkovich_McCoy_Orrick}. In \cite{AndrewsBerkovich}, Andrews--Berkovich applied their trinomial analogue of Bailey's lemma to the infinite hierarchy that includes \eqref{BMO_GG1_bosonic} and found related new summation formulas. We can do the same using a more recent summation formula on $T_0$ trinomials due to Berkovich and the second author \cite[Theorem~3.5]{BU_elementary}. 

\begin{theorem}\label{thm_BaileyP010} For non-negative integer $L$, we have
\begin{align}\nonumber
    \sum_{i,j,k\geq 0} q^{i^2+j^2 + k^2} {L \brack i}_{q^2} &{ i - k\brack j}_{q^2} {j\brack k}_{q^2} \\\label{eq_BaileydP010}&= \sum_{j=-\infty}^\infty \left((-1)^j q^{20j^2+j} {2L\brack L-4j}_{q^2} 
    +(-1)^j q^{20j^2+9j+1}{2L\brack L-4j-1}_{q^2} 
    \right).
\end{align}
\end{theorem}

As $L\rightarrow\infty$, applying the Jacobi Triple Product identity, we get the following.
\begin{corollary}\label{cor_P010JTP}
    \[\sum_{i,j,k\geq 0} \frac{q^{i^2+j^2 + k^2}}{(q^2;q^2)_i}{ i - k\brack j}_{q^2} {j\brack k}_{q^2}= \frac{(q^{40};q^{40})_\infty}{(q^2;q^2)_\infty} \left((q^{19},q^{21};q^{40})_\infty+q(q^{11},q^{29};q^{40})_\infty)\right).\]
\end{corollary}

Moreover, now \eqref{eq_BaileydP010} can be further iterated by another special case of the Bailey lemma due to Berkovich and the second author \cite[Theorem 2.1]{BerkUncu}. 

\begin{theorem} For non-negative integers $L$ and $p$
\begin{align}\nonumber
    \nonumber\sum_{\substack{i,j,k\geq 0\\m_p\geq \dots\geq m_{1}\geq 0}} &\frac{q^{2m_p^2+2m_{p-1}^2+\dots +2m_1^2+i^2+j^2 + k^2} (q^2;q^2)_{2L}}{(q^2;q^2)_{L-m_p}(q^2;q^2)_{m_p - m_{p-1}}\dots (q^2;q^2)_{m_2-m_1}(q^2;q^2)_{2m_1}} {m_1 \brack i}_{q^2} { i - k\brack j}_{q^2} {j\brack k}_{q^2} \\ \label{eq_double_baileyd_P010}&= \sum_{j=-\infty}^\infty \left((-1)^j q^{(20+32p)j^2+j} {2L\brack L-4j}_{q^2} 
    +(-1)^j q^{(20+32p)j^2+(9+16p)j+1+2p}{2L\brack L-4j-1}_{q^2} 
    \right).
\end{align}
\end{theorem}

The Jacobi Triple Product identity as $L\rightarrow\infty$ implies \eqref{eq_double_baileyd_P010lim}.

In the spirit of \cite{BerkUncu,BU_elementary,UW}, we can map $q\mapsto 1/q$ for the dual identity of the one in Theorem~\ref{thm_BaileyP010}.

\begin{theorem}\label{thm_BaileyP010recip} For a non-negative integer $L$, we have
\begin{align}\nonumber
    q^{2L^2}\sum_{i,j,k\geq 0} q^{i^2 - 2 i j + j^2 + k^2 - 2 i L}& {L \brack i}_{q^2} { i - k\brack j}_{q^2} {j\brack k}_{q^2} \\\label{eq_BaileydP010recip}&= \sum_{j=-\infty}^\infty \left((-1)^j q^{12j^2-j} {2L\brack L-4j}_{q^2} 
    +(-1)^j q^{12 j^2+7j+1}{2L\brack L-4j-1}_{q^2} 
    \right).
\end{align}
\end{theorem}

We would like to take the limit $L\rightarrow\infty$ in \eqref{eq_BaileydP010recip}. However, the appearance of the $ L$- dependent terms on the left-hand side $q$-term requires us to study this a little. On the left-hand side of \eqref{eq_BaileydP010recip}, let $m= i-j$, then the quadratic that appears as at the exponent of $q$ in the summand becomes $2L^2 - 2iL + m^2 + k^2$, with natural summation bounds of $m$ being from $0$ to $i$. Then, once $L\rightarrow\infty$, only the terms when $i=L$ survive. Moreover, these terms' limit is the same as the limit $L\rightarrow\infty$ of \eqref{BMO_GG1_bosonic}'s left-hand side. We also recall that \eqref{BMO_GG1_bosonic}'s left-hand side is equal to the product generating function of the first G\"{o}llnitz--Gordon identity (Theorem~\ref{thm_GG}). When $L\rightarrow\infty$, the right-hand side of \eqref{eq_BaileydP010recip} can be summed with the Jacobi Triple Product identity as in \eqref{cor_P010JTP}.

\begin{corollary}\label{cor_GG_dissec}
    \begin{equation}\label{eq_GG_dissec}
        \frac{1}{(q,q^4,q^7;q^8)_\infty} = \frac{(q^{24};q^{24})}{(q;q)_\infty}\left( (q^{11},q^{13};q^{24})_\infty + q (q^{5},q^{19};q^{24})_\infty\right).
    \end{equation}
\end{corollary}

Corollary~\ref{cor_GG_dissec} can also be recovered by \cite{AlladiGG}, where Alladi derives an equivalent identity (\eqref{eq_GG_dissec} with $q\mapsto-q^2$) using a dissection of Euler's Pentagonal Number Theorem.

Here we can apply \cite[Theorem 2.1]{BerkUncu} to Theorem~\ref{thm_BaileyP010recip} and see the entire infinite hierarchy that grows from the seed identity of Theorem~\ref{thm_BaileyP010recip}.

\begin{theorem}
    For non-negative integers $L$ and $p$
\begin{align}\nonumber
    \nonumber\sum_{\substack{i,j,k\geq 0\\m_p\geq \dots\geq m_{1}\geq 0}} &\frac{q^{2m_p^2+2m_{p-1}^2+\dots +2m_2^2+4m_1^2-2im_1+i^2-2ij+j^2 + k^2} (q^2;q^2)_{2L}}{(q^2;q^2)_{L-m_p}(q^2;q^2)_{m_p - m_{p-1}}\dots (q^2;q^2)_{m_2-m_1}(q^2;q^2)_{2m_1}} {m_1 \brack i}_{q^2} { i - k\brack j}_{q^2} {j\brack k}_{q^2} \\ \label{eq_double_baileyd_P010_dual}&= \sum_{j=-\infty}^\infty \left((-1)^j q^{(12+32p)j^2-j} {2L\brack L-4j}_{q^2} 
    +(-1)^j q^{(12+32p)j^2+(7+16p)j+1+2p}{2L\brack L-4j-1}_{q^2} 
    \right).
\end{align}
\end{theorem}

As $L\rightarrow\infty$, we get \eqref{eq_double_baileyd_P010_dual_lim}.

\section{Profile $000$}\label{sec:000}

The diagram of a partition in $\text{DSPP}_{000}$ with at most $4n$ parts can be of the following shapes.
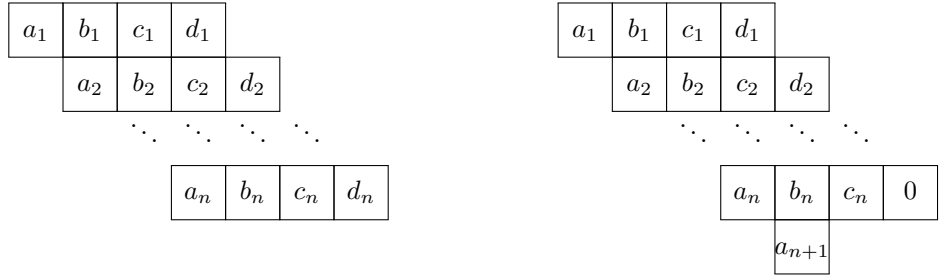
\begin{center}
\begin{ytableau}
a_1 & b_1 & c_1 & d_1\\
\none & a_2 & b_2 & c_2 & d_2\\
\none & \none & \none[\ddots] & \none[\ddots] & \none[\ddots] & \none[\ddots]\\
\none & \none & \none & a_{n} & b_{n} & c_n & d_n
\end{ytableau}
\hspace{2cm}
\begin{ytableau}
a_1 & b_1 & c_1 & d_1\\
\none & a_2 & b_2 & c_2 & d_2\\
\none & \none & \none[\ddots] & \none[\ddots] & \none[\ddots] & \none[\ddots]\\
\none & \none & \none & a_{n} & b_{n} & c_n &0 \\
\none & \none & \none & \none & a_{n+1}
\end{ytableau}
\end{center}
The first diagram represents those with diagonal length bounded by $n$, that is, the set $\text{DSPP}_{000}(n)$. Let $\text{DSPP}'_{000}(n)$ be the set of double shifted partitions with profile ${000}$ such that $a_{n+1}>0$ and $b_{n+1}=c_{n+1}=d_{n}=0$. Finally, let $\text{DSPP}^{T}_{000}(n)$ be the set of double shifted partitions with profile ${000}$ that have at most $4n$ parts. Then $\text{DSPP}_{000}(n)\cap\text{DSPP}'_{000}(n)=\emptyset$ and $\text{DSPP}^{T}_{000}(n)=\text{DSPP}_{000}(n)\cup\text{DSPP}'_{000}(n)$. Let 
$$F_{000}(n):=\sum_{\Lambda\in\text{DSPP}_{000}(n)}q^{|\Lambda|},\quad F'_{000}(n):=\sum_{\Lambda\in\text{DSPP}'_{000}(n)}q^{|\Lambda|}\quad\text{and}\quad F^{T}_{000}(n):=\sum_{\Lambda\in\text{DSPP}^{T}_{000}(n)}q^{|\Lambda|}$$
be their generating functions, respectively.
\begin{theorem}
The crude form of $F_{000}(n)$ is
\begin{equation}
\begin{split}
F_{000}(n)=&\underset{\geq}{\Omega} \frac{1}{(1-q\lambda_{1,1}\mu_{1,1})}\left(1-\frac{q\lambda_{1,2}\mu_{1,2}}{\lambda_{1,1}}\right)^{-1}\left(1-\frac{q\lambda_{1,3}\mu_{1,3}}{\lambda_{1,2}}\right)^{-1}\prod_{i=1}^{n}\left(1-\frac{q\lambda_{i,i+3}\mu_{i,i+3}}{\lambda_{i,i+2}}\right)^{-1}\\
&\times\prod_{i=2}^{n}\left(1-\frac{q\lambda_{i,i}\mu_{i,i}}{\mu_{i-1,i}}\right)^{-1}\left(1-\frac{q\lambda_{i,i+1}\mu_{i,i+1}}{\lambda_{i,i}\mu_{i-1,i+1}}\right)^{-1}\left(1-\frac{q\lambda_{i,i+2}\mu_{i,i+2}}{\lambda_{i,i+1}\mu_{i-1,i+2}}\right)^{-1}.
\end{split}
\end{equation}
\end{theorem}

\begin{theorem}
The initial values of $F_{000}(n)$ are
\begin{equation*}
F_{000}(1)= \frac{1}{(q;q)_4}, 
\end{equation*}
\begin{equation*}
F_{000}(2)=\frac{1}{(q;q)_8}(1+q^3+q^4+q^5+q^8).    
\end{equation*}
\end{theorem}

\begin{theorem}\label{F000Rec}
For any $n\geq1$,
\begin{equation}\label{eq:F000Rec}
F_{000}(n)=\frac{1+q^{4n-3}-q^{8n-7}}{(q^{4n-3};q)_4}F_{000}(n-1)+\frac{1+q^{4n-3}}{(q^{4n-3};q)_{4}}F'_{000}(n-1),
\end{equation}
\begin{equation}\label{eq:F'000Rec}
\begin{split}
F'_{000}(n)=&\frac{(q^{4n}+q^{4n-1})(1+q^{4n-3}-q^{8n-7})+q^{4n-2}(1-q^{4n-4})}{(q^{4n-3};q)_{4}}F_{000}(n-1)\\
&+\frac{(q^{4n}+q^{4n-1})(1+q^{4n-3})+q^{4n-2}}{(q^{4n-3};q)_4}F'_{000}(n-1).
\end{split}
\end{equation}  
\end{theorem}
The proof is similar to Theorem~\ref{thm:F010Rec}, so we skip it. Let
$P_{000}(n):=(q;q)_{4n}F_{000}(n)$, $P'_{000}(n):=(q;q)_{4n}F'_{000}(n)$ and $P^{T}_{000}(n):=(q;q)_{4n}F^{T}_{000}(n)$, then Theorem~\ref{F000Rec} implies the following.
\begin{theorem}\label{thm:RecP000}
For $n\geq1$,
\begin{equation}\label{eq:P000Rec}
P_{000}(n)=(1+q^{4n-3}-q^{8n-7})P_{000}(n-1)+(1+q^{4n-3})P'_{000}(n-1),   
\end{equation}   
\begin{equation}\label{P'000Rec}
\begin{split}
P'_{000}(n)=&((q^{4n}+q^{4n-1})(1+q^{4n-3}-q^{8n-7})+q^{4n-2}(1-q^{4n-4}))P_{000}(n-1)\\
&+((q^{4n}+q^{4n-1})(1+q^{4n-3})+q^{4n-2})P'_{000}(n-1).
\end{split} 
\end{equation}
\end{theorem}

Once again, we can uncouple the recurrences \eqref{eq:P000Rec} and \eqref{P'000Rec} and see the uncoupled recurrences that help us with the automated proofs. For example, we have

\begin{align}
     \nonumber q^{2 + 8 n} (1 + q^{5 + 4 n}) P_{000}(
   n) - (1 + q^{3 + 4 n}) (1 + q^{1 + 4 n} + q^{2 + 4 n} + q^{
    4 + 4 n} &+ q^{5 + 4 n} + q^{6 + 8 n}) P_{000}(n+1)\\ \label{P010_rec_uncoupled}&+ (1 + q^{1 + 4 n}) P_{000}( n+2) = 0.
\end{align}

Looking at the initial terms, we find the following formulas, which we prove using the tools in \cite{qMultiSum}.

\begin{theorem}\label{thm:P000_fermionic_rep}
    For $n\geq 0$, \begin{align}
        \label{P000_fermionic_rep} P_{000}(n) &= \sum_{j,k\geq 0} q^{j^2 + k^2 + 2j} {2 n - k-2\brack j}_{q^2} {j\brack k}_{q^2}.
    \end{align}
\end{theorem}

A formula for $P'_{000}(n)$ can be recovered from \eqref{eq:P000Rec}, but we did not observe a shorter formula as in the previous section.

The $P_{000}(n)$ is the generating function for the number of partitions that satisfy the second G\"ollnitz--Gordon partition theorem, where all parts are bounded by $4n-3$. Hence, in the limit $n\rightarrow\infty$, by \eqref{G2_prod}, we prove \eqref{F000_product}.

Moreover, the following polynomial relation, which yields the second G\"ollnitz--Gordon identity directly, can be recovered from \cite{Berkovich_McCoy_Orrick}.

\begin{theorem} For every non-negative integer $L$,
\begin{equation}
    \label{BMO_GG2_bosonic}
    \sum_{j,k\geq 0} q^{j^2 + k^2+2j} {L - k-1\brack j}_{q^2} {j\brack k}_{q^2} = \sum_{j=-\infty}^\infty (-1)^j q^{4j^2+3j} \left(T_0(L,4j+1,q^2) +  T_0(L,4j+2,q^2) \right).
\end{equation}
\end{theorem}

 We can, once again, utilize the summation formulas \cite[Theorem~3.5]{BU_elementary} and \cite[Theorem 2.1]{BerkUncu} by Berkovich and the second author.

\begin{theorem}\label{thm_baileyd000} For a non-negative integers $L$ and $p$
\begin{align}\nonumber
    \nonumber\sum_{\substack{i,j,k\geq 0\\m_p\geq \dots\geq m_{1}\geq 0}} &\frac{q^{2m_p^2+2m_{p-1}^2+\dots +2m_1^2+i^2+j^2 + k^2+2j} (q^2;q^2)_{2L}}{(q^2;q^2)_{L-m_p}(q^2;q^2)_{m_p - m_{p-1}}\dots (q^2;q^2)_{m_2-m_1}(q^2;q^2)_{2m_1}} {m_1 \brack i}_{q^2} { i - k-1\brack j}_{q^2} {j\brack k}_{q^2} \\ \label{eq_double_baileyd_P000}&= q^{1+2p} \sum_{j=-\infty}^\infty \left((-1)^j q^{(20+32p)j^2+(11+16p)j} {2L\brack L-4j-1}_{q^2} \right.\\ \nonumber &\hspace{4cm} \left.
    +(-1)^j q^{(20+32p)j^2+(19+32p)j+3+6p}{2L\brack L-4j-2}_{q^2} 
    \right).
\end{align}
\end{theorem}

As $L\rightarrow\infty$, Theorem~\ref{thm_baileyd000} implies \eqref{eq_double_baileyd_P010lim}.

We can once again look at the $q\mapsto 1/q$ image and grow an infinite hierarchy from that point. The analogous identity to Theorem~\ref{thm_BaileyP010recip} is the following.

\begin{theorem}\label{thm_BaileyP000recip} For a non-negative integer $L$, we have
\begin{align}\nonumber
    q^{2L^2}\sum_{i,j,k\geq 0} &q^{i^2 - 2 i j + j^2 + k^2 - 2 i L} {L \brack i}_{q^2} { i - k-1\brack j}_{q^2} {j\brack k}_{q^2} \\\label{eq_BaileydP010recip}&= q\sum_{j=-\infty}^\infty \left((-1)^j q^{12j^2+5j} {2L\brack L-4j-1}_{q^2} 
    +(-1)^j q^{12 j^2+13j+3}{2L\brack L-4j-2}_{q^2} 
    \right).
\end{align}
\end{theorem}

\begin{theorem}
    For non-negative integers $L$ and $p$
\begin{align}\nonumber
    \nonumber\sum_{\substack{i,j,k\geq 0\\m_p\geq \dots\geq m_{1}\geq 0}} &\frac{q^{2m_p^2+2m_{p-1}^2+\dots +2m_2^2+4m_1^2-2im_1+i^2-2ij+j^2 + k^2} (q^2;q^2)_{2L}}{(q^2;q^2)_{L-m_p}(q^2;q^2)_{m_p - m_{p-1}}\dots (q^2;q^2)_{m_2-m_1}(q^2;q^2)_{2m_1}} {m_1 \brack i}_{q^2} { i - k-1\brack j}_{q^2} {j\brack k}_{q^2} \\ \label{eq_double_baileyd_P010_dual}&=q^{1+2p} \sum_{j=-\infty}^\infty \left((-1)^j q^{(12+32p)j^2+(5+16p)j} {2L\brack L-4j}_{q^2} \right.\\ \nonumber&\hspace{5cm}\left.
    +(-1)^j q^{(12+32p)j^2+(13+32p)j+3+6p}{2L\brack L-4j-1}_{q^2} 
    \right).
\end{align}
\end{theorem}

The identity \eqref{eq_double_baileyd_P000_dual_lim} can easily be seen by taking the $L\rightarrow\infty$ limit of \eqref{eq_double_baileyd_P010_dual} and elementary manipulations. The $p=0$ case's limit can also be extracted from \cite{AlladiGG}. The left-hand side of the $p=0$ case's limit is the same as the limit of the left-hand side of \eqref{BMO_GG2_bosonic} with an extra $q$. We write this limit and the dissection next.

\begin{corollary}
   \[ \frac{1}{(q^3,q^4,q^5;q^8)} = \frac{(q^{24};q^{24})_\infty}{(q^2;q^2)_\infty}\left( (q^{7}, q^{17};q^{24})_\infty - q^{2}(q, q^{23};q^{24})_\infty
    \right).\]
\end{corollary}

\section{Profile $001$}\label{sec:001}

The diagram of a partition in $\text{DSPP}_{001}$ with at most $4n$ parts can be of the following shapes.
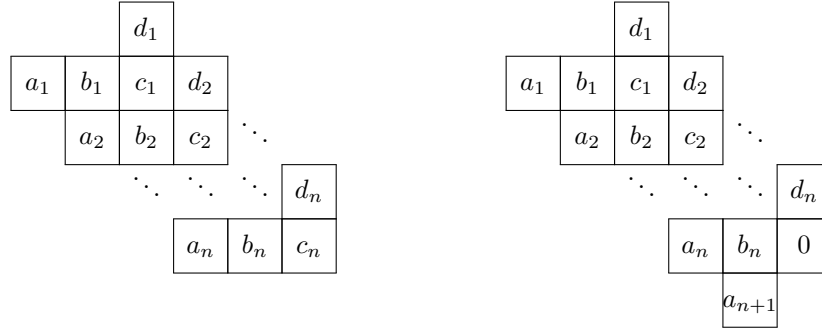
\begin{center}
\begin{ytableau}
\none & \none & d_1 \\
a_1 & b_1 & c_1 & d_2\\
\none & a_2 & b_2 & c_2 & \none[\ddots]\\
\none & \none &
\none[\ddots] & \none[\ddots] & \none[\ddots] & d_{n}\\
\none & \none & \none &  a_{n} & b_{n} & c_n
\end{ytableau}
\hspace{2cm}
\begin{ytableau}
\none & \none & d_1 \\
a_1 & b_1 & c_1 & d_2\\
\none & a_2 & b_2 & c_2 & \none[\ddots]\\
\none & \none &
\none[\ddots] & \none[\ddots] & \none[\ddots] & d_{n}\\
\none & \none & \none &  a_{n} & b_{n} & 0\\
\none & \none & \none & \none & a_{n+1}
\end{ytableau}
\end{center}
Let $\text{DSPP}_{001}(n)$ be the set of double shifted partitions of profile ${001}$ with diagonal length bounded by $n$, that is, $a_{n+1}=b_{n+1}=c_{n+1}=d_{n+1}=0$. And let $\text{DSPP}'_{001}(n)$ be the set of double shifted partitions with profile ${001}$ such that $a_{n+1}>0$ and $b_{n+1}=c_{n}=d_{n+1}=0$. Finally, let $\text{DSPP}^{T}_{001}(n)$ be the set of double shifted partitions with profile ${001}$ that have at most $4n$ parts. Then $\text{DSPP}_{001}(n)\cap\text{DSPP}'_{001}(n)=\emptyset$ and $\text{DSPP}^{T}_{001}(n)=\text{DSPP}_{001}(n)\cup\text{DSPP}'_{001}(n)$. Let 
$$F_{001}(n):=\sum_{\Lambda\in\text{DSPP}_{001}(n)}q^{|\Lambda|},\quad F'_{001}(n):=\sum_{\Lambda\in\text{DSPP}'_{001}(n)}q^{|\Lambda|}\quad\text{and}\quad F^{T}_{001}(n):=\sum_{\Lambda\in\text{DSPP}^{T}_{001}(n)}q^{|\Lambda|}$$
be their generating functions, respectively.

\begin{theorem}
The crude form of $F_{001}(n)$ is
\begin{equation}
\begin{split}
F_{001}(n)=&\underset{\geq}
{\Omega}\frac{1}{(1-q\lambda_{2,1}\mu_{2,1})(1-q\lambda_{1,3}\mu_{1,3})}\left(1-\frac{q\lambda_{2,2}\mu_{2,2}}{\lambda_{2,1}}\right)^{-1}\prod_{n=1}^{\infty}\left(1-\frac{q\lambda_{i+1,i+2}\mu_{i+1,1+2}}{\lambda_{i+1,i+1}\mu_{i,i+2}}\right)^{-1}\\
&\times\prod_{i=2}^{n}\left(1-\frac{q\lambda_{i+1,i}\mu_{i+1,i}}{\mu_{i,i}}\right)^{-1}\left(1-\frac{q\lambda_{i+1,i+1}\mu_{i+1,i+1}}{\lambda_{i+1,i}\mu_{i,i+1}}\right)^{-1}\left(1-\frac{q\lambda_{i,i+2}\mu_{i,i+2}}{\lambda_{i,i+1}}\right)^{-1}.
\end{split}
\end{equation}
\end{theorem}

\begin{theorem}
The initial values of $F_{001}(n)$ are
$$F_{001}(1)=\frac{1}{(q;q)_{4}}(1+q+q^{2}),$$
$$F_{001}(2)=\frac{1}{(q;q)_{8}}(1+q+q^2+q^3+q^4+2q^5+3q^6+2q^7+q^8+q^9+q^{10}+q^{11}+q^{12}).$$
\end{theorem}
\begin{theorem}\label{F001Rec}
For any $n\geq1$,
\begin{equation}\label{eq:F001Rec}
F_{001}(n)=\frac{1+q^{4n-2}+q^{4n-3}}{(q^{4n-3};q)_4}F_{001}(n-1)+\frac{1+q^{4n-2}}{(q^{4n-3};q)_4}F'_{001}(n-1),    
\end{equation}
\begin{equation}\label{eq:F'001Rec}
F'_{001}(n)=\frac{q^{4n-1}+q^{4n}+q^{8n-3}+q^{8n-2}}{(q^{4n-3};q)_4}F_{001}(n-1)+\frac{q^{4n-1}+q^{4n}+q^{8n-2}}{(q^{4n-3};q)_4}F'_{001}(n-1).   
\end{equation}
\end{theorem}
Again, the proof is similar to Theorem~\ref{thm:F010Rec}, so we skip it. Consequently, let $P_{001}(n):=(q;q)_{4n}F_{001}(n)$, we have the following.
\begin{theorem}
For any $n\geq1$,
\begin{equation}\label{eq:P001Rec}
P_{001}(n)=(1+q^{4n-2}+q^{4n-3})P_{001}(n-1)+(1+q^{4n-2})P'_{001}(n-1),    
\end{equation}
\begin{equation}\label{eq:P'001Rec}
P'_{001}(n)=(q^{4n-1}+q^{4n}+q^{8n-3}+q^{8n-2})P_{001}(n-1)+(q^{4n-1}+q^{4n}+q^{8n-2})P'_{001}(n-1).     
\end{equation}
\end{theorem}

Similar to the previous cases, we can uncouple the recurrences. Here we get \begin{align}
 \nonumber q^{4 + 8 n} (1 + q^{6 + 4 n}) P_{001}(
   n) - (1 + q^{4 + 4 n}) (1 + q^{2 + 4 n} + q^{3 + 4 n} + q^{
    5 + 4 n} &+ q^{6 + 4 n} + q^{8 + 8 n}) P_{001}(n+1)\\ \label{P001_rec_uncoupled}&+ (1 + q^{2 + 4 n}) P_{001}( n+2) = 0    
\end{align}

Here we have the following closed form generating function for $P_{001}(n)$.

\begin{theorem}\label{thm:P001_fermionic_rep}
    For $n\geq 0$, \begin{align}
        \label{P001_fermionic_rep} P_{001}(n) &= \sum_{j,k\geq 0} q^{j^2 + k^2 + j} {2 n - k-1\brack j}_{q^2} {j+1\brack k}_{q^2}.
    \end{align}
\end{theorem}

The right-hand side of \eqref{P001_fermionic_rep} was shown to be the generating function for the number of first little G\"ollnitz identities' partitions that satisfy the gap conditions with the additional bound that all parts $\leq 4n-2$ in \cite[Theorem 4.6]{doublesums}. Therefore, in the limit $n\rightarrow\infty$, the little G\"ollnitz theorem's product size proves \eqref{DSPP_littleG1}.

\section{Profile $100$}\label{sec:100}

The diagram of a partition in $\text{DSPP}_{100}$ with at most $4n$ parts can be of the following shapes.
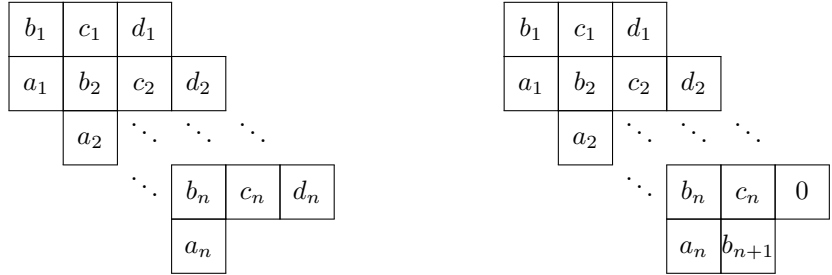
\begin{center}
\begin{ytableau}
b_1 & c_1 & d_1\\
a_1 & b_2 & c_2 & d_2\\
\none & a_2 & \none[\ddots] & \none[\ddots] & \none[\ddots]\\
\none & \none & \none[\ddots] & b_{n} & c_{n} & d_{n}\\
\none & \none & \none & a_{n} 
\end{ytableau}
\hspace{2cm}
\begin{ytableau}
b_1 & c_1 & d_1\\
a_1 & b_2 & c_2 & d_2\\
\none & a_2 & \none[\ddots] & \none[\ddots] & \none[\ddots]\\
\none & \none & \none[\ddots] & b_{n} & c_{n} & 0\\
\none & \none & \none & a_{n} & b_{n+1} 
\end{ytableau}
\end{center}
Let $\text{DSPP}_{100}(n)$ be the set of double shifted partitions of profile ${100}$ with at most $n$ entries on each diagonal, that is, $a_{n+1}=b_{n+1}=c_{n+1}=d_{n+1}=0$. And let $\text{DSPP}'_{100}(n)$ be the set of double shifted partitions with profile ${100}$ such that $b_{n+1}>0$ and $a_{n+1}=c_{n+1}=d_{n}=0$. Finally, let $\text{DSPP}^{T}_{100}(n)$ be the set of double shifted partitions with profile ${100}$ that have at most $4n$ parts. Then $\text{DSPP}_{100}(n)\cap\text{DSPP}'_{100}(n)=\emptyset$ and $\text{DSPP}^{T}_{100}(n)=\text{DSPP}_{100}(n)\cup\text{DSPP}'_{100}(n)$. Let 
$$F_{100}(n):=\sum_{\Lambda\in\text{DSPP}_{100}(n)}q^{|\Lambda|},\quad F'_{100}(n):=\sum_{\Lambda\in\text{DSPP}'_{100}(n)}q^{|\Lambda|}\quad\text{and}\quad F^{T}_{100}(n):=\sum_{\Lambda\in\text{DSPP}^{T}_{100}(n)}q^{|\Lambda|}$$
be their generating functions, respectively.

\begin{theorem}
The crude form for $F_{100}(n)$ is
\begin{equation}
\begin{split}
F_{100}(n)=&\underset{\geq}{\Omega}\frac{1}{(1-q\lambda_{1,1}\mu_{1,1})}\left(1-\frac{q\lambda_{1,2}\mu_{1,2}}{\lambda_{1,1}}\right)^{-1}\prod_{n=2}^{n}\left(1-\frac{q\lambda_{i,i}\mu_{i,i}}{\lambda_{i,i-1}\mu_{i-1,i}}\right)^{-1}\left(1-\frac{q\lambda_{i,i+1}\mu_{i,i+1}}{\lambda_{i,i}\mu_{i-1,i+1}}\right)^{-1}\\
&\times\prod_{i=1}^{n}\left(1-\frac{q\lambda_{i+1,i}\mu_{i+1,i}}{\mu_{i,i}}\right)^{-1}\left(1-\frac{q\lambda_{i,i+2}\mu_{i,i+2}}{\lambda_{i,i+1}}\right)^{-1}.
\end{split}
\end{equation}
\end{theorem}

\begin{theorem}
The initial values of $F_{100}(n)$ are
$$F_{100}(1)=\frac{1}{(q;q)_{4}}(1+q^2+q^3),$$
$$F_{100}(2)=\frac{1}{(q;q)_{8}}(1+q^2+q^3+q^4+q^5+2q^6+2q^7+q^8+2q^9+2q^{10}+q^{11}+q^{12}+q^{13}).$$
\end{theorem}

\begin{theorem}\label{thm:F100Rec}
The following recurrences hold for $F_{100}(n)$ and $F'_{100}(n)$.
\begin{equation}
F_{100}(n)=\frac{1+q^{4n-1}+q^{4n-2}}{(q^{4n-3};q)_4}F_{100}(n-1)+\frac{1+q+q^{4n-1}+q^{4n-2}}{(q^{4n-3};q)_4}F'_{100}(n-1),
\end{equation}
\begin{equation}
F'_{100}(n)=\frac{q^{4n}(1+q^{4n-2}}{(q^{4n-3};q)_4}F_{100}(n-1)+\frac{q^{4n}(1+q+q^{4n-2})}{(q^{4n-3};q)_4}F'_{100}(n-1).
\end{equation}
\end{theorem}
The proof is similar to Theorem \ref{thm:F010Rec}, so we skip it. Recall that $P_{100}(n):=(q;q)_{4n}F_{100}(n)$ and $P'_{100}(n):=(q;q)_{4n}F'_{100}(n)$. This implies the following.
\begin{theorem}\label{thm:P100Rec}
\begin{equation}
P_{100}(n)=(1+q^{4n-1}+q^{4n-2})P_{100}(n-1)+(1+q+q^{4n-1}+q^{4n-2})P'_{100}(n-1), 
\end{equation}
\begin{equation}
P'_{100}(n)=q^{4n}(1+q^{4n-2})P_{100}(n-1)+q^{4n}(1+q+q^{4n-2})P'_{100}(n-1).    
\end{equation}
\end{theorem}

As in the previous cases, we can find a closed formula for the generating function $P_{100}(n)$. However, this time the formula is more involved, suggesting that this diagonal construction with $P_{100}(n)$ and $P'_{100}(n)$ might not be the most natural one.

\begin{theorem}\label{thm:P100_fermionic_rep}
    For $n\geq 0$, \begin{align}
        \label{P100_fermionic_rep} P_{100}(n) &= \sum_{j,k\geq 0} q^{j^2 + k^2 + j} {2 n - k-1\brack j}_{q^2} {j\brack k}_{q^2}+ q^{4n-1} \sum_{j,k\geq 0} q^{j^2 + k^2 + j} {2 n - k-2\brack j}_{q^2} {j\brack k}_{q^2}.
    \end{align}
\end{theorem}

The formula in Theorem~\ref{thm:P100_fermionic_rep} comes in two pieces. As $L\rightarrow\infty$ the second part vanishes and the first part converges to the generating function for the number of partitions that satisfy the gap conditions of the second little G\"ollnitz identity by \cite[Theorem 4.6]{doublesums}. This proves \eqref{DSPP_littleG2}.

\section{The palindromic pattern}\label{sec:Palindromity}

In this section, we discuss the palindromic pattern in the polynomials. Giving a polynomial $p(x)$, let $\text{deg}(p(x))$ be the degree of $p(x)$. A polynomial is called palindromic if
$$p(x)=x^{\text{deg}(p(x))}p(x^{-1}).$$
Note that for profile $010$, $000$, and $001$, the initial values we gave in previous sections are all palindromic. This turns out to be true for all $n$.
\begin{theorem}\label{thm:Pc Palindromic}
For any $n\geq0$, $P_{010}(n)$, $P_{000}(n)$, and $P_{001}(n)$ are all palindromic. Namely,
$$P_c(n;q)=q^{\text{deg}(P_c(n;q))}P_c(n;q)^{-1}$$
for $c=010$, $000$, and $001$.
\end{theorem}

\begin{proof}
We first treat the profile $010$. By Theorem~\ref{thm:P010_fermionic_rep},
$$P_{010}(n;q)=\sum_{j,k\geq0}q^{j^2+k^2}{2n-k\brack j}_{q^2}{j\brack k}_{q^2}.$$
Note that the $q$-binomial coefficient
$${n\brack m}_{q}:=\frac{(q;q)_n}{(q;q)_m(q;q)_{m-m}}$$
is always palindromic and has degree $mn$, so
$${m+n\brack m}_{q}=q^{mn}{m+n\brack m}_{q^{-1}}.$$
Meanwhile, for each term in the summation of $P_{010}(n;q)$, the degree is
$$(j^2+k^2)+2(2n-k-j)j+2(j-k)k=4nj-j^2-k^2.$$
This will reach the largest value when $j=2n$ and $k=0$, so $\text{deg}(P_{010}(n;q))=4n^2$.
Now,
\begin{align*}
q^{\text{deg}(P_{010}(n;q))}P_{010}(n;q^{-1})=&q^{4n^2}\sum_{j,k\geq0}q^{-j^2-k^2}{2n-k\brack j}_{q^{-2}}{j\brack k}_{q^{-2}}\\
=&\sum_{j,k\geq 0} q^{4n^2-j^2  -k^2-2(2n-k-j)j-2(j-k)k} {2 n - k\brack j}_{q^2} {j\brack k}_{q^2}\\
=&\sum_{j,k\geq 0} q^{4n^2-4nj+j^2+k^2}  {2 n - k\brack j}_{q^2} {j\brack k}_{q^2}\\
=&\sum_{j,k\geq 0} q^{(2n-j)^2+k^2}  {2 n - k\brack j}_{q^2} {j\brack k}_{q^2}\\
=&\sum_{m,k\geq0}q^{m^2+k^2}{2 n - k\brack 2n-m}_{q^2} {2n-m\brack k}_{q^2}\\
=&\sum_{m,k\geq0}q^{m^2+k^2}\frac{(q^2;q^2)_{2n-k}}{(q^2;q^2)_{2n-m}(q^2;q^2)_{m-k}}\cdot\frac{(q^2;q^2)_{2n-m}}{(q^2;q^2)_{k}(q^2;q^2)_{2n-m-k}}\\
=&\sum_{m,k\geq0}q^{m^2+k^2}\frac{(q^2;q^2)_{2n-k}}{(q^2;q^2)_{m-k}}\cdot\frac{1}{(q^2;q^2)_{k}(q^2;q^2)_{2n-m-k}}\\
=&\sum_{m,k\geq0}q^{m^2+k^2}\frac{(q^2;q^2)_{2n-k}}{(q^2;q^2)_m(q^2;q^2)_{2n-m-k}}\cdot\frac{(q^2;q^2)_m}{(q^2;q^2)_{k}(q^2;q^2)_{m-k}}\\
=&\sum_{m,k\geq0}q^{m^2+k^2}{2n-k\brack m}_{q^2}{m\brack k}_{q^2}= P_{010}(n;q).
\end{align*}
So, we have shown that $P_{010}(n)$ is palindromic. For $P_{000}(n)$ and $P_{001}(n)$, it suffices to apply the same argument on \eqref{P000_fermionic_rep} and \eqref{P001_fermionic_rep}. So, we finish the proof.
\end{proof}

Unfortunately, the palindromic pattern does not hold for $P_{100}(n;q)$; we shall fix this by cutting the diagram in a different way. Recall that we considered the following diagrams, which correspond to $\text{DSPP}_{100}(n)$ and $\text{DSPP}'_{100}(n)$, respectively.
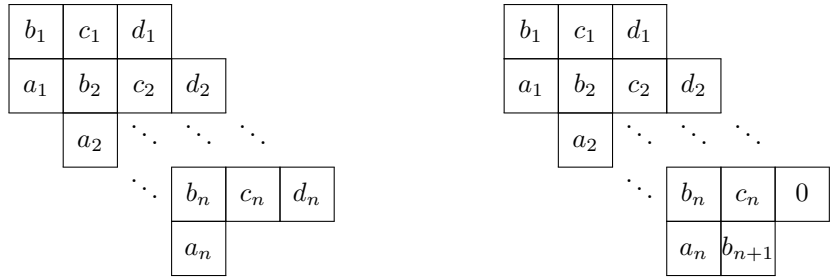
\begin{center}
\begin{ytableau}
b_1 & c_1 & d_1\\
a_1 & b_2 & c_2 & d_2\\
\none & a_2 & \none[\ddots] & \none[\ddots] & \none[\ddots]\\
\none & \none & \none[\ddots] & b_{n} & c_{n} & d_{n}\\
\none & \none & \none & a_{n} 
\end{ytableau}
\hspace{2cm}
\begin{ytableau}
b_1 & c_1 & d_1\\
a_1 & b_2 & c_2 & d_2\\
\none & a_2 & \none[\ddots] & \none[\ddots] & \none[\ddots]\\
\none & \none & \none[\ddots] & b_{n} & c_{n} & 0\\
\none & \none & \none & a_{n} & b_{n+1} 
\end{ytableau}
\end{center}
We further define $\text{DSPP}''_{100}(n)$ to be the set of partitions in $\text{DSPP}_{100}$ with $a_{n+1}=b_{n+2}=c_{n+1}=d_{n}=0$. In other words, $\text{DSPP}''_{100}(n)$ contains partitions represented by the same diagram $\text{DSPP}'_{100}(n)$, except that $b_{n+1}$ is no longer necessarily greater than $0$. Let 
$$F''_{100}(n):=F''_{100}(n;q):=\sum_{\Lambda\in\text{DSPP}''_{100}(n)}q^{|\Lambda|}\quad\text{and}\quad P''_{100}(n):=P''_{100}(n):=(q;q)_{4n}F''_{100}(n).$$
It is clear that $F''_{100}(n)=F'_{100}(n)/q^{4n}$, hence $P''_{100}(n)=P'_{100}(n)/q^{4n}$. By Partition Analysis, we have the crude form and initial values as follows.
\begin{theorem}
\begin{equation}
\begin{split}
F''_{100}(n)=&\underset{\geq}{\Omega}\frac{1}{(1-q\lambda_{1,1}\mu_{1,1})}\left(1-\frac{q\lambda_{1,2}\mu_{1,2}}{\lambda_{1,1}}\right)^{-1}\prod_{i=2}^{n}\left(1-\frac{q\lambda_{i,i}\mu_{i,i}}{\lambda_{i,i-1}\mu_{i-1,i}}\right)^{-1}\left(1-\frac{q\lambda_{i,i+1}\mu_{i,i+1}}{\lambda_{i,i}\mu_{i-1,i+1}}\right)^{-1}\\
&\times\prod_{i=1}^{n}\left(1-\frac{q\lambda_{i+1,i}}{\mu_{i,i}}\right)^{-1}\left(1-\frac{q\mu_{i,i+2}}{\lambda_{i,i+1}}\right)^{-1}.
\end{split}
\end{equation}  
\end{theorem}

\begin{theorem}
\begin{equation}
F''_{100}(1)=\frac{1}{(q;q)_4}(1+q^2),  
\end{equation}
\begin{equation}
F''_{100}(2)=\frac{1}{(q;q)_8}(1+q^2+q^3+q^4+q^5+2 q^6+q^7+q^8+q^9+q^{10}+q^{12}).   
\end{equation}
\end{theorem}

By Theorem~\ref{thm:P100Rec}, we have the following.
\begin{theorem} The sequence $P''_{100}(n)$ satisfies the following recurrence.
   \begin{align}\nonumber q^{4 + 8 n} (1 + q^{6 + 4 n}) P''_{100}(n) - (1 + q^{4 + 4 n}) (1 + q^{2 + 4 n} + q^{3 + 4 n} + q^{     5 + 4 n} &+ q^{6 + 4 n} + q^{8 + 8 n}) P''_{100}(n+1) \\\label{eq_P''rec}&+ (1 + q^{2 + 4 n}) P''_{100}(n+2) = 0\end{align}
\end{theorem}
Note that \eqref{eq_P''rec} is the same as \eqref{P001_rec_uncoupled}. This is expected as the gap conditions on these partition theorems are the same, and only the initial conditions differ. This indicates that this is a more natural truncation of this generating over the one in Section~\ref{sec:100}.

By solving this recurrence, we get the closed form for $P''_{100}(n)$.
\begin{theorem}
\begin{equation}\label{Pdoubleprimeexp}
    P''_{100}(n)= \sum_{j,k\geq 0} q^{j^2 + k^2 + j} {2 n - k-1\brack j}_{q^2} {j\brack k}_{q^2},
\end{equation}
and $P''_{100}(n)$ is palindromic for all $n$.
\end{theorem}

The palindromic property can be proved in the same way as in Theorem \ref{thm:Pc Palindromic}. Observe that the right-hand side of \eqref{Pdoubleprimeexp} is the first sum in \eqref{P100_fermionic_rep}. No need for a correction term is needed, as in \eqref{P100_fermionic_rep}, and the terms are palindromic, as in the expressions of $P_{010}(n)$, $P_{000}(n)$, and $P_{001}(n)$. This further indicates that the use of $P''_{100}(n)$ is a more natural truncation of these objects.

\section{Connection with linear partitions}\label{sec:LinearPartition}

Recall that in Section \ref{sec:Intro} we introduced the interpretations for the sum sides of G\"ollnitz-Gordon and Little G\"ollnitz identities in terms of linear partitions with gap conditions. It turns out that those polynomials $P_c(n)$ we found are also the generating functions of these linear partitions with bounded part size. To see this, we first need the bivariable generating functions for the aforementioned partitions that keep track of both the weight and the largest part.

Giving a partition $\lambda=(\lambda_1,\lambda_2,\ldots,\lambda_{\ell})$, define the refined weight as 
$$x^{\lambda}:=x_1^{\lambda_1}x_2^{\lambda_2}x_3^{\lambda_3}\cdots x_{\ell}^{\lambda_{\ell}}.$$
For the sake of compactness, we also define
$$X_i:=\left\{\begin{array}{cc}
   x_1x_2\cdots x_i  & \text{for}\quad i\geq1,\\
   1 & \text{otherwise.} 
\end{array}\right.$$
Using Partition Analysis, Andrews and Paule~\cite{AndrewsPauleParity}, and the first author~\cite{Li} proved the following.
\begin{theorem}\label{thm:AndrewsPaule}[Andrews and Paule, 2025]
\begin{equation}
\sum_{\lambda\in G_1}x^{\lambda}=1+\sum_{n=1}^{\infty}\frac{X_1^{2}X_2^{2}\cdots X_{n-1}^{2}X_n(1+X_1)(1+X_1X_2)\cdots(1+X_{n-1}X_{n})}{(1-X_1^2)(1-X_2^2)\cdots(1-X_n^2)}.
\end{equation}   
\end{theorem}
\begin{theorem}\label{thm:Li}[Li, 2025]
\begin{equation}\label{eq:RefinedG1'}
\sum_{\lambda\in G_1'}x^{\lambda}=1+\sum_{n=1}^{\infty}\frac{X_1^{2}X_2^{2}\cdots X_n^2(1+X_1)(1+X_1X_2)\cdots(1+X_{n-2}X_{n-1})(1+X_{n-1}/X_n)}{(1-X_1^2)(1-X_2^2)\cdots(1-X_n^2)},
\end{equation}
\begin{equation}\label{eq:RefinedG2'}
\sum_{\lambda\in{G}'_2}x^{\lambda}=1+\sum_{n=1}^{\infty}\frac{X_1^2X_2^2\cdots X_{n}^{2}(1+X_1)(1+X_1X_2)\cdots(1+X_{n-1}X_n)}{(1-X_1^2)(1-X_2^2)\cdots(1-X_n^2)}.
\end{equation}    
\end{theorem}
Though the generating function for $G_2$ is missing, readers familiar with Partition Analysis can easily verify the following. 
\begin{theorem}
\begin{equation}
\sum_{\lambda\in{G}_2}x^{\lambda}=1+\sum_{n=1}^{\infty}\frac{X_1^{2}X_2^{2}\cdots X_{n-1}^{2}X_n^{3}(1+X_1)(1+X_1X_2)\cdots(1+X_{n-1}X_{n})}{(1-X_1^2)(1-X_2^2)\cdots(1-X_n^2)}.
\end{equation}    
\end{theorem}
The proof follows the same process as Theorem \eqref{thm:AndrewsPaule} and Theorem \ref{thm:Li}, so we skip it. Let $x_1\to zq$ and $x_i\to q$ for $i>1$, which implies $X_i\to zq^{i}$ for $i\geq1$, we have the following bivariable generating functions.
\begin{corollary}
Let $\text{lar}(\lambda)$ be the largest part of $\lambda$,
\begin{equation}\label{eq:G1Lar}
\sum_{\lambda\in G_1}z^{\text{lar}(\lambda)}q^{|\lambda|}=(1+zq)\left(1+\sum_{n=1}^{\infty}\frac{z^{2n-1}q^{n^{2}}(-z^2q^3;q^2)_{n-1}}{(z^2q^2;q^2)_n}\right),
\end{equation}
\begin{equation}\label{eq:G2Lar}
\sum_{\lambda\in G_2}z^{\text{lar}(\lambda)}q^{|\lambda|}=(1+zq)\left(1+\sum_{n=1}^{\infty}\frac{z^{2n+1}q^{n^{2}+2n}(-z^2q^3;q^2)_{n-1}}{(z^2q^2;q^2)_n}\right),
\end{equation}
\begin{equation}\label{eq:G'1Lar}
\sum_{\lambda\in G'_1}z^{\text{lar}(\lambda)}q^{|\lambda|}=(1+q^{-1})(1+zq)\left(1+\frac{z^2q^2}{1-z^2q^2}+\sum_{n=2}^{\infty}\frac{z^{2n}q^{n^2+n}(-z^2q^{3};q^2)_{n-2}}{(z^2q^2;q^2)_{n}}\right),
\end{equation}
\begin{equation}\label{eq:G'2Lar}
\sum_{\lambda\in G'_2}z^{\text{lar}(\lambda)}q^{|\lambda|}=(1+zq)\left(1+\sum_{n=1}^{\infty}\frac{z^{2n}q^{n^2+n}(-z^2q^3;q^2)_{n-1}}{(z^2q^2;q^2)_{n}}\right).
\end{equation} 
\end{corollary}

Let $G_{1}(n)$ generating function for partitions in $G_{1}$ with parts bounded by $4n-1$, $G_{2}(n)$ be the generating function for partitions in $G_{2}$ with parts bounded by $4n-3$, $G'_{1}(n)$ generating function for partitions in $G'_{1}$ with parts bounded by $4n-2$, and $G'_{2}(n)$ be the generating function for partitions in $G'_{2}$ with parts bounded by $4n-1$. Also, we define the modified $q$-binomial coefficients as
$${n 
\brack m}'_q=:\left\{\begin{array}{cc}
    1 & \text{if $n<0$ and $m=0$}, \\
    {n\brack m} & \text{otherwise.} 
\end{array}\right.$$
\begin{theorem}
\begin{equation}\label{eq:G1n}
\begin{split}
G_1(n)=&\sum_{i+j+k\leq 2n}q^{j^{2}+2i+k^{2}+2k}{j+i-1\brack i}'_{q^{2}}{j-1\brack k}'_{q^{2}}\\
&\hspace{3cm}+\sum_{i+j+k\leq 2n-1}q^{j^{2}+2i+k^{2}+2k+1}{j+i-1\brack i}'_{q^{2}}{j-1\brack k}'_{q^{2}},   
\end{split}
\end{equation} 
\begin{equation}\label{eq:G2n}
\begin{split}
G_2(n)=&\sum_{i+j+k\leq 2n-2}q^{j^{2}+2j+2i+k^{2}+2k}{j+i-1\brack i}'_{q^{2}}{j-1\brack k}'_{q^{2}}\\
&\hspace{3cm}+\sum_{i+j+k\leq 2n-3}q^{j^{2}+2j+2i+k^{2}+2k+1}{j+i-1\brack i}'_{q^{2}}{j-1\brack k}'_{q^{2}}, 
\end{split}  
\end{equation}
\begin{equation}\label{eq:G'1n}
\begin{split}
G'_1(n)=&(1+q^{-1})\left(\sum_{i+j+k\leq 2n-1}q^{j^{2}+j+2i+k^{2}+2k}{j+i-1\brack i}'_{q^{2}}{j-2\brack k}'_{q^{2}}\right.\\
&\hspace{3cm}\left.+\sum_{i+j+k\leq 2n-2}q^{j^{2}+j+2i+k^{2}+2k+1}{j+i-1\brack i}'_{q^{2}}{j-2\brack k}'_{q^{2}}\right)-\frac{1}{q},
\end{split}    
\end{equation}

\begin{equation}\label{eq:G'2n}
\begin{split}
G'_2(n)=&\sum_{i+j+k\leq 2n-1}q^{j^{2}+j+2i+k^{2}+2k}{j+i-1\brack i}'_{q^{2}}{j-1\brack k}'_{q^{2}}\\
&\hspace{3cm}+\sum_{i+j+k\leq 2n-2}q^{j^{2}+j+2i+k^{2}+2k+1}{j+i-1\brack i}'_{q^{2}}{j-1\brack k}'_{q^{2}}.
\end{split}   
\end{equation}
\end{theorem}

\begin{proof}
We shall take \eqref{eq:G1n} as an example; the other three can be proved in the same way. We will need the following finite versions of $q$-binomial theorem \cite[p. $36$, $(3.3.6)$, $(3.3.7)$]{A}.
\begin{equation}\label{eq:FQBinom1}
 (z;q)_{n}=\sum_{k=0}^{n}(-1)^kz^{k}q^{\binom{k}{2}}{n\brack k}_{q},   
\end{equation}
\begin{equation}\label{eq:FQBinom2}
 \frac{1}{(z;q)_{n}}=\sum_{k=0}^{n}z^{k}{n+k-1\brack k}_{q}.  
\end{equation}
By \eqref{eq:G1Lar},
\begin{align}
\nonumber\sum_{\lambda\in G_1}z^{\text{lar}(\lambda)}q^{|\lambda|}=&(1+zq)\left(1+\sum_{n=1}^{\infty}\frac{z^{2n-1}q^{n^{2}}(-z^2q^3;q^2)_{n-1}}{(z^2q^2;q^2)_n}\right)\\
\nonumber
=&(1+zq)\left(1+\sum_{n=1}^{\infty}\frac{z^{2n-1}q^{n^{2}}}{(z^2q^2;q^2)_n}\sum_{k=0}^{n-1}z^{2k}q^{k^2+2k}{n-1\brack k}_{q^2}\right)\\
\nonumber
&\hspace{4cm}\text{(by \eqref{eq:FQBinom1} with $z\to -z^2q^3$ and $q\to q^2$)}\\
\nonumber
=&(1+zq)\left(1+\sum_{n=1}^{\infty}z^{2n-1}q^{n^{2}}\sum_{i=0}^{n}z^{2i}q^{2i}{n+i-1\brack i}_{q^2}\sum_{k=0}^{n-1}z^{2k}q^{k^2+2k}{n-1\brack k}_{q^2}\right)\\
\nonumber
&\hspace{4cm}\text{(by \eqref{eq:FQBinom2} with $z\to z^2q^2$ and $q\to q^2$)}\\
\label{GG1_z_genF}
=&(1+zq)\sum_{i,j,k\geq0}z^{2i+2j+2k-1}q^{j^2+2i+k^2+2k}{j+i-1\brack i}'_{q^2}{j-1\brack k}'_{q^2}.
\end{align}
Since in $G_1(n)$ we want the largest part to be bounded by $4n-1$, by collecting those terms in the series, we finish the proof of \eqref{eq:G1n}. \eqref{eq:G2n}, \eqref{eq:G'1n}, and \eqref{eq:G'2n} follow the same argument, so we finish the proof. 
\end{proof}

Notice that by bounding the exponent of $z$ in \eqref{GG1_z_genF} (and its analogs), one gets explicit sum representations for the generating functions related to G\"ollnitz--Gordon and little G\"ollnitz partition identities with the extra condition on the largest part. These are exactly the same generating functions that were represented in the earlier sections. Explicitly, with the bounds chosen, the right-hand sides of \eqref{eq:G1n}, \eqref{eq:G2n}, \eqref{eq:G'1n}, and \eqref{eq:G'2n} coincide with the right-hand sides of $P_{010}(n)$ in \eqref{P010_fermionic_rep}, $P_{000}(n)$ in \eqref{P000_fermionic_rep}, $P_{001}(n)$ in \eqref{P001_fermionic_rep}, and $P_{100}(n)$ in \eqref{P100_fermionic_rep}, respectively. We write the first one to highlight this connection:

\begin{theorem}\label{lastthm}For non-negative integer $n$, we have
    \begin{align*}\sum_{j,k\geq 0} q^{j^2 + k^2} {2n - k\brack j}_{q^2} {j\brack k}_{q^2}&=\sum_{i+j+k\leq 2n}q^{j^{2}+2i+k^{2}+2k}{j+i-1\brack i}'_{q^{2}}{j-1\brack k}'_{q^{2}}\\
&\hspace{3cm}+\sum_{i+j+k\leq 2n-1}q^{j^{2}+2i+k^{2}+2k+1}{j+i-1\brack i}'_{q^{2}}{j-1\brack k}'_{q^{2}}.\end{align*}
\end{theorem}

\section{Outlook}\label{sec:Conclusions}

The extra structure cylindric partitions bring over skew double shifted plane partitions became highly visible in this paper once again. Although we followed the same approach of \cite{LiUncu}, both the recurrences and solving them proven to be more difficult. In fact, the polynomial identities that imply the G\"ollnitz--Gordon identities we identified were due to Berkovich--McKoy--Orrick \cite{Berkovich_McCoy_Orrick}, and we were unable to find a polynomial sequence that converges to the product side of the little G\"ollnitz identities. We plan to pursue identifying polynomial sequences that converge to the product side of little G\"ollnitz identities with the use of some basic hypergeometric summations in the future.

We were able to identify both the refined multi-sum (fermionic) and the bilateral-single-sum (bosonic) representations of the identities when we were studying cylindric partitions with two-element profiles using MacMahon's partition analysis, i.e., when we imposed bounds on the number of non-zero elements on the rows of the cylindric partitions. That led to the observation of how the general bosonic generating functions would look like for all 2-element profiles. We do not have the same structure here. Although we looked at larger profiles, we did not observe a relation between the generating functions of DSPPs with different profiles. However, for the cases we could identify, the summation formula \cite[Theorem 2.1]{BerkUncu} still yielded beautiful infinite hierarchies: Theorem~\ref{main_thm} and its finite versions presented in Sections~\ref{sec:010} and~\ref{sec:000}. 

Proving the formulas for the generating functions for the DSPPs required us to do so combinatorially by studying the objects and later comparing them with the symbolically calculated recurrences (using creative telescoping, see \cite{HolonomicFunctions}). This is no different than in \cite{LiUncu}. However, DSPPs required us to approach the problem with a coupled system first and then to uncouple this system. This is not how we dealt with cylindric partitions. There, for a fixed profile, we were directly calculating the recurrences. We suspect that if we study cylindric partitions with 2-element profiles with a coupled system approach, we can understand the general shape of the recurrences without the need to fix a profile first. That can allow us to directly turn all the conjectures of that paper into theorems without the need for any extra work. We plan to pursue that.

We see Theorem~\ref{lastthm} and its analogs through the generating functions. It is desirable to see a bijective proof that would show that the DSPP's counted by one side of the identities become and transform bijectively into the G\"ollnitz--Gordon and little G\"ollnitz identities related partitions with gap conditions. That would also allow us to possibly identify statistics in the DSPPs and partitions which can later be used to refine and generalize these identities.

We also plan to look into interesting subclasses of cylindric partitions and DSPPs in the near future using the same approach and by putting extra bounds. We expect to discover many interesting polynomial identities as different representations of generating functions, and many new infinite hierarchies of identities to root from them.


\begin{thebibliography}{99}

\bibitem{qFuncs} J. Ablinger and A.K. Uncu, \texttt{qFunctions }\textit{- A Mathematica package for $q$-series and partition theory applications}, J. Symbolic Comput. 107 (2021), 145-166.

\bibitem{AlladiGG} K. Alladi, \textit{Some new observations on the Göllnitz-Gordon and Rogers-Ramanujan identities}. Trans. Amer. Math. Soc., 347(3) (1995), 897–914.

\bibitem{AAG} K. Alladi, G. E. Andrews, and B. Gordon, \textit{Generalizations and refinements of a partition theorem of Göllnitz} J. Reine Angew. Math., 460 (1995). 165–188.

\bibitem{A} G.~E.~Andrews, {\it The Theory of Partitions}, Cambridge University Press (1998).

\bibitem{AndrewsI}
G. E. Andrews, {\it MacMahon's partition analysis I: The lecture hall partition theorem},  Progress in Mathematics, vol 161. Birkhäuser Boston. https://doi.org/10.1007/978-1-4612-4108-9\_1.

\bibitem{AndrewsBerkovich} G. E. Andrews, and A. Berkovich, \textit{A trinomial analogue of Bailey’s lemma and $N = 2$ superconformal invariance}. Communications in Mathematical Physics, 192(2) (1998), 245–260. https://doi.org/10.1007/s002200050298.

\bibitem{AndrewsPauleVIII} G. E. Andrews, P. Paule, and A. Riese,  {\it MacMahon's Partition Analysis: VIII. Plane Partition Diamonds}. Advances in Applied Mathematics, 27(2-3) (2001) pp.231-242.

\bibitem{AndrewsPauleXIII} G. E. Andrews, and P. Paule. {\it MacMahon's partition analysis XIII: Schmidt type partitions and modular forms}. Journal of Number Theory 234 (2022): 95-119.

\bibitem{AndrewsPauleParity}
G. E. Andrews and P. Paule, \emph{MacMahon's partition analysis XV: Parity}, J. Symbolic Comput., \textbf{127} (2025), 102351.

\bibitem{OmegaPackage}
G. E. Andrews, P. Paule, and R. Axel, \emph{MacMahon's partition analysis: the Omega package}, European J. Combin. 22 (2001), no. 7, 887–904.

\bibitem{Berkovich_McCoy_Orrick} A. Berkovich, B. M. McCoy and W. P. Orrick, \textit{Polynomial identities, indices and duality for the $N=1$ superconformal model $SM(2,4\nu)$}, J. Statist Phys. \textbf{83} (1996), no. 5-6, 795-837.

\bibitem{BU_elementary} A. Berkovich, A. K. Uncu, \textit{Elementary polynomial identities involving $q$-trinomial coefficients}. In K. Alladi, B. C. Berndt, P. Paule, J. A. Sellers, and A. J. Yee (Eds.), George E. Andrews 80 Years of Combinatory Analysis (pp. 119–127). Birkhäuser. (2021) https://doi.org/10.1007/978-3-030-57050-7\_11.

\bibitem{BerkUncu} A. Berkovich, A. K. Uncu, {\it New infinite hierarchies of polynomial identities related to the Capparelli partition theorems}. Journal of Mathematical Analysis and Applications 506.2 (2022): 125678.

\bibitem{BridgesUncu} W. Bridges, A. K. Uncu, {\it Weighted cylindric partitions} Journal of Algebraic Combinatorics volume 56, pages 1309–1337 (2022).

\bibitem{CDU} S. Corteel, J. Dousse, and A. K. Uncu. {\it Cylindric partitions and some new $A_2$ Rogers--Ramanujan identities}. Proceedings of the American Mathematical Society 150.2 (2022): 481-497.

\bibitem{HX} G.-N. Han and H. Xiong, {\it Skew doubled shifted plane partitions: calculus and asymptotics}, Amer. Inst. Math. Sci. {\bf 29} (2021),  1841--1857.

\bibitem{HolonomicFunctions}  C. Koutschan, {\it Advanced Applications of the Holonomic Systems Approach}, RISC, Johannes Kepler University,
Linz. PhD Thesis. September 2009.

\bibitem{Li}
R. Li, {\it Partition analysis and the new little Göllnitz identities}, preprint
arXiv:2510.23233.

\bibitem{LiUncu}
R. Li, and A. K. Uncu, {\it A MacMahon analysis view of cylindric partitions}, Ramanujan J. 68 (2025), no. 3, Paper No. 71, 41 pp.

\bibitem{MacMahon}
P. A. MacMahon, {\it Combinatorial analysis, Volumes I and II}, Vol. 137. American Mathematical Society, 2001.

\bibitem{qMultiSum} A. Riese, \textit{qMultiSum - A Package for Proving q-Hypergeometric Multiple Summation Identities}, Journal of Symbolic Computation 35 (2003), 349-376.

\bibitem{doublesums}A. K. Uncu, {\it On double sum generating functions in connection with some classical partition theorems}. Disc. Math J. Volume 344, Issue 11 (2021), 112562.

\bibitem{UW} A. K. Uncu, and W. Zudilin, \textit{Reflecting (on) the modulo 9 Kanade–Russell (conjectural) identities}, S\'{e}m. Lothar. Combin. 85 (2021), Art. B85e, 17 pages.


\end{thebibliography}
\end{document}